\documentclass[abstract]{scrartcl}

\usepackage[utf8]{inputenc}
\usepackage[T1]{fontenc}
\usepackage{mathtools}
\usepackage{amsfonts}
\usepackage{amssymb}
\usepackage{amsthm}

\usepackage{soul}

\usepackage{graphicx}
\usepackage[english]{babel}
\usepackage{mathtools}

\usepackage{hyperref}
\hypersetup{ 
	colorlinks,
	linkcolor={blue!60!black},
	citecolor={blue!60!black},
	urlcolor={blue!60!black},
	linktoc=page}

\usepackage[numbers]{natbib}
\usepackage{authblk} 

\usepackage{todonotes}

\newcommand{\cE}{\mathcal{E}}
\newcommand{\cL}{\mathcal{L}}
\newcommand{\cO}{\mathcal{O}}
\newcommand{\der}[2]{\frac{\partial {#1}}{\partial {#2}}}

\renewcommand{\d}{\mathrm{d}}

\newcommand{\N}{\mathbb{N}}
\newcommand{\R}{\mathbb{R}}
\newcommand{\Z}{\mathbb{Z}}

\renewcommand{\mod}{\mathrm{mod}}

\DeclareMathOperator{\D}{D}

\newtheorem{theorem}{Theorem}
\newtheorem{definition}{Definition}
\newtheorem{proposition}{Proposition}

\theoremstyle{definition}
\newtheorem{example}{Example}

\title{Contact variational integrators}

\author[1]{Mats Vermeeren}
\author[2]{Alessandro Bravetti}
\author[3]{Marcello Seri}
\affil[1]{Technische Universität Berlin, Germany, \authorcr\normalsize \texttt{vermeeren@math.tu-berlin.de}}
\affil[2]{Centro de Investigación en Matemáticas (CIMAT), Guanajuato, Mexico, \authorcr\normalsize \texttt{alessandro.bravetti@cimat.mx}}
\affil[3]{Bernoulli Institute for Mathematics, Computer Science and Artificial Intelligence, \authorcr\small Groningen, The Netherlands, 
\authorcr\normalsize \texttt{m.seri@rug.nl}}

\date{}

\begin{document}

\maketitle
\abstract{
\noindent
We present geometric numerical integrators for contact flows that stem from a discretization of Herglotz’ variational principle. First we show that the resulting discrete map is a contact transformation and that any contact map can be derived from a variational principle. Then we discuss the backward error analysis of our variational integrators, including the construction of a modified Lagrangian. Throughout the paper we use the damped harmonic oscillator as a benchmark example to compare our integrators to their symplectic analogues.


}

\section{Introduction}

The last few years have seen a rise in importance of the field of contact geometry. 
As the theory gets more relevant to scientific applications, 
there is an increasing demand for the development of numerical integrators preserving the contact structure.

Contact geometry appears in fluid dynamics \cite{etnyre2000contact,kholodenko2013applications,roulstone1994hamiltonian}, thermodynamics~\cite{bravetti2018contact, grmela2014contact, grmela1997dynamics, Mrugala1978, vanderschaft2018geomthermo},
statistical physics \cite{bravetti2016thermostat}, statistics \cite{betancourt2014adiabatic,goto2018contact}, quantum mechanics \cite{bravetti2017exact,ciaglia2018contact,fitzpatrick2011geometric,herczeg2018contact,rajeev2008quantization}, gravity~\cite{lazo2017action}, information geometry \cite{barbaresco2013information,barbaresco2016geometric,goto2016contact}, shape dynamics \cite{sloan2018dynamical}, biology \cite{bravetti2018thermodynamics, guha2017generalized}, optimal control \cite{jozwikowski2016contact,ohsawa2015contact}, and integrable systems \cite{Boyer,jovanovic2012noncommutative,jovanovic2015contact,sergyeyev2018integrable,sergyeyev2018new}. One of the applications that has recently attracted a lot of attention is the classical mechanics of dissipative systems \cite{abraham1978foundations,Arnold,bravetti2017contact,deleon2017cos,de2019singular,gaset2019new,liu2018contact}, on which we will focus in this paper.

Contact geometry can be thought of as an odd-dimensional analogue of symplectic geometry. A contact manifold is a pair $(M,\xi)$ where $M$ is an $(2n+1)$-dimensional manifold and $\xi\subset{TM}$ is a contact structure, that is, a maximally non-integrable distribution of hyperplanes. Locally, such distribution is given by the kernel of a one form $\eta$ satisfying $\eta \wedge (\d \eta)^n \neq 0$ (see e.g.~\cite{geiges2008introduction} for more details). The 1-form $\eta$ is called the contact form. Note that
if we multiply $\eta$ by a non-vanishing function we obtain another 1-form giving rise to the same contact structure $\xi$. 
This means that in order to preserve $\xi$ we can act on the 1-form by conformal transformations. Thus we must keep in mind that $\eta$ is just a representative element in an equivalence class of 1-forms describing the same $\xi$.

Once we fix $\eta$, Darboux's theorem for contact manifolds states that for any point on $M$ there exists a neighbourhood with local coordinates $(x_1,\ldots,x_n,p_1,\ldots,p_n,z)$ such that the contact 1-form can be written as
\[ \eta =  \d z-\sum_i p_i \,\d x_i . \]
Throughout the paper we will write $\d z-p\,\d x$ as a short form for $\eta$.

Moreover, given $\eta$, to any smooth function $H:M\rightarrow\mathbb{R}$ we can associate a contact Hamiltonian vector field $X_H$, defined by
\[ \mathcal{L}_{X_H} \eta = f_H \eta \qquad \text{and} \qquad \eta(X_H) = - H, \]
where $\mathcal{L}$ is the Lie derivative, $f_H=-R_{\eta}(H)$ and $R_{\eta}$ is the Reeb vector field corresponding to $\eta$~\cite{geiges2008introduction}.
In Darboux coordinates the flow of $X_H$ is given by
\[
\begin{dcases}
    \dot{x} = \der{H}{p} \\
    \dot{p} = -\der{H}{x} - p \der{H}{z} \\
    \dot{z} = p \der{H}{p} - H ,
\end{dcases}
\]
where $x = (x_1,\ldots,x_n)$, $p = (p_1,\ldots,p_n)$ and the standard scalar product is assumed where two vectors are multiplied. The flow of a contact Hamiltonian system preserves the contact structure, but it does not preserve the Hamiltonian. Instead we have
\[ \frac{\d H}{\d t} = -H \der{H}{z} . \]
For example a Hamiltonian of the form $H = \frac{1}{2} p^2 + V(x) + \alpha z$ leads to the equations of motion of a damped mechanical system:
\[
\begin{dcases}
    \dot{x} = p \\
    \dot{p} = -V'(x) - \alpha p \\
    \dot{z} = p^2 - H .
\end{dcases}
\]

Another remarkable similarity with standard symplectic Hamiltonian systems is the fact that contact Hamiltonian systems have an associated variational principle, which is due to Herglotz~\cite{georgieva2003generalized,herglotz1930vorlesungen} (see also~\cite{cannarsa2018herglotz,wang2016implicit}), and a corresponding theory of generating functions~\cite{bravetti2017contact}.

Geometric integrators for contact Hamiltonian systems have been studied in~\cite{Feng2010} by 
exploiting their symplectification and the corresponding generating functions. However, a variational approach is missing.
So far,~\cite{Feng2010} has received little attention, most likely because the authors did not discuss any physically relevant examples.

In this paper we present a natural way to develop numerical integrators for contact systems by exploiting Herglotz' variational principle. Our result furnishes a variational scheme to integrate contact Hamiltonian systems in such a way that the contact structure is preserved. Furthermore, in analogy with the theory of symplectic numerical integrators, we find that modified equations for our method are again contact systems. This suggests that the numerical results will remain very close to a ``nearby'' contact system for very long times.

The purpose of this paper is to lay the theoretical groundwork of contact integrators and to show their promise with some simple numerical experiments. To keep the discussion direct and self-contained, the main results will be presented only for contact systems without an explicit time dependence, with a section showing how the method is easily extended to general contact systems by means of an explanatory example. Furthermore, some interesting discussions and developments are postponed to future works. These include the extension to a more general sub-Riemannian setting and the comparison with~\cite{Feng2010} and other known approaches to some physically relevant examples.

It is important to remark at this point that not every kind of dissipative system can be written with a contact Hamiltonian of the form above, as the geometry underlying the construction will enforce some structure. For more details on this, we refer the readers to the thorough investigation in \cite{ciaglia2018contact,de2019singular,gaset2019new}. Nevertheless, the contact Hamiltonian structure allows to describe a large number of physically relevant systems, including most of the ones from the literature cited at the beginning of this introduction.

The paper is structured as follows. In section \ref{sec:hvp} we set the stage introducing Herglotz' variational principle and some of its relevant properties. In section \ref{sec:dhvp} we develop the central idea of the paper defining a contact integrator obtained from a discretization of Herglotz' variational principle. In section \ref{sec:bea} we study the modified equations for the contact integrators introduced in section \ref{sec:dhvp}, showing that up to truncation errors, the numerical solutions are interpolated by contact systems. In section \ref{sec:tdep} we present an example to show how the ideas can be extended in a straightforward fashion to systems with an explicit time dependence. Finally, in section \ref{sec:num} we illustrate by numerical examples how our integrators perform on contact systems in comparison to symplectic integrators.

\section{Herglotz' variational principle}\label{sec:hvp}

The usual variational principle in mechanics looks for a curve $x:[0,T] \rightarrow Q$ in configuration space $Q$, such that the action integral
\begin{equation}\label{classical-action}
S = \int_0^T \cL(t,x(t),\dot{x}(t)) \, \d t
\end{equation}
is critical with respect to variations of $x$ that vanish at the endpoints, where $\cL: \R \times TQ \rightarrow \R$ is a given Lagrange function. Herglotz \cite{herglotz1930vorlesungen} generalized this variational principle by defining the action in terms of a differential equation instead of an integral.
\begin{definition}
Let $\cL: \R \times TQ \times \R \rightarrow \R$. Given a curve $x:[0,T] \rightarrow Q$, define the function $z:[0,T] \rightarrow \R$ by an initial condition $z(0) = z_0$ and the differential equation
\begin{equation}\label{z-differential}
\dot{z}(t) = \cL(t,x(t),\dot{x}(t),z(t)) .
\end{equation}
The curve $x$ is a \emph{solution to Herglotz' variational principle} with initial condition $z_0$ if every variation of $x$ that vanishes at the boundary of $[0,T]$ leaves the action $z(T)$ invariant.
\end{definition}
If $\cL$ does not depend on $z$, then the differential equation \eqref{z-differential} is solved by the integral \eqref{classical-action} and Herglotz' variational principle reduces to the classical variational principle. A modern discussion of Herglotz' variational principle can be found for example in \cite{georgieva2003generalized}.

\begin{proposition}\label{prop-gEL}
A (sufficiently regular) curve $x$ is a solution to Herglotz' variational principle {if and only if} it satisfies the \emph{generalized Euler-Lagrange equations}
\begin{equation}\label{generalized-EL}
\der{\cL}{x} - \frac{\d}{\d t} \der{\cL}{\dot{x}} + \der{\cL}{z}\der{\cL}{\dot{x}} = 0 ,
\end{equation} 
where $z$ is given in terms of $x$ by Equation \eqref{z-differential}.
\end{proposition}

Note that the Euler-Lagrange equations are not linear in $\cL$, hence they are not invariant under scaling of the Lagrangian $\cL$.

\begin{proof}[Proof of Proposition \ref{prop-gEL}]
Consider an arbitrary variation $\delta x$ of $x$, vanishing at the endpoints, and the corresponding induced variation $\delta z$ of $z$. Since the initial condition $z(0) = z_0$ is independent of $x$ we have $\delta z(0) = 0$. From Equation \eqref{z-differential} it follows that
\[ \delta \dot{z} = \der{\cL}{x} \delta x + \der{\cL}{\dot{x}} \delta \dot{x} + \der{\cL}{z} \delta z .\]
If we set 
\[ A(t) = \der{\cL}{x} \delta x + \der{\cL}{\dot{x}} \delta \dot{x} \qquad \text{and} \qquad B(t) = \int_0^t \der{\cL}{z}(\tau) \,\d \tau, \]
this differential equation reads
\[  \delta \dot{z}(t) = A(t) + \frac{\d B(t)}{\d t} \delta z \]
and its solution is
\[ \delta z(t) = e^{B(t)} \left[ \int_0^t A(\tau) e^{-B(\tau)} \,\d \tau + \delta z(0) \right]. \]
Plugging in the expression for $A$ and noting that $\frac{\d B}{\d t} = \der{\cL}{z}$ we find
\begin{align}
\delta z(T) &= e^{B(T)} \left[ \int_0^T  \left( \der{\cL}{x} \delta x + \der{\cL}{\dot{x}} \delta \dot{x} \right) e^{-B(\tau)} \, \d \tau  + \delta z(0) \right] \notag \\
\begin{split}
&= e^{B(T)} \bigg[ \int_0^T  \left( \der{\cL}{x} - \frac{\d}{\d t} \der{\cL}{\dot{x}} + \der{\cL}{z}\der{\cL}{\dot{x}} \right)  \delta x \,  e^{-B(\tau)} \, \d \tau \\
&\hspace{2cm} + \der{\cL}{\dot{x}}(T) \delta x(T) \,  e^{-B(T)}- \der{\cL}{\dot{x}}(0) \delta x(0)  + \delta z(0) \bigg] .
\end{split}\label{general-variation}
\end{align}
The boundary terms vanish because $\delta x(0) = \delta x(T) = \delta z(0) = 0$. Since $\delta x$ is otherwise arbitrary, the action $z(T)$ is critical if and only if Equation \eqref{generalized-EL} holds.
\end{proof}

If the classical variational principle is satisfied on the interval $[0,T]$ it is automatically satisfied on any subinterval.  For the Herglotz variational principle this property is not obvious from the definition, but it still follows from the generalized Euler-Lagrange equations.

\begin{proposition}
If $x:[0,T] \rightarrow Q$ solves the Herglotz variational principle with initial condition $z_0$, then for any interval $[a,b] \subset [0,T]$, the restriction $x|_{[a,b]}$ solves the Herglotz variational principle with initial condition $z(a)$.
\end{proposition}
\begin{proof}
If $x$ is critical on $[0,T]$ then the generalized Euler-Lagrange equations are satisfied everywhere on this interval. In particular, they hold on $[a,b]$, hence $x$ is critical on $[a,b]$.
\end{proof}

In the following we will assume that the Lagrangian is regular, i.e.\@ $\left| \der{^2 \cL}{\dot{x}^2} \right| \neq 0$.
Then the generalized Euler-Lagrange equations can be written explicitly as a second order ODE. Together with the evolution of $z$ we find the system of ODEs
\begin{align*}
\ddot{x} &= \left(\der{^2 \cL}{\dot{x}^2} \right)^{-1} \left( \der{\cL}{x} - \der{^2 \cL}{\dot{x} \partial x} \dot{x} - \der{^2 \cL}{\dot{x} \partial z} \cL + \der{\cL}{z}\der{\cL}{\dot{x}} \right) , \\
\dot{z} &= \cL .
\end{align*}

An important aspect of the Herglotz variational principle is that the energy is not conserved (unless the Lagrangian is independent of $z$). Instead we find a differential equation governing its evolution.

\begin{proposition}
If the Lagrangian does not explicitly depend on time, then the energy $E = \der{\cL}{\dot{x}} \dot{x} - \cL$ satisfies the differential equation
\begin{equation}\label{energy-diffeqn}
\dot{E} = \der{\cL}{z} E 
\end{equation}
\end{proposition}
\begin{proof}
Consider a uniform time-shift of the critical curve $x$ and the function $z$. Then $\delta x = \dot{x}$ and $\delta z = \dot{z}$. If the Lagrangian does not explicitly depend on time, it follows from Equation \eqref{general-variation} that 
\[ \der{\cL}{\dot{x}}(t) \dot{x}(t) - \dot{z}(t) = e^{B(t)} \left( \der{\cL}{\dot{x}}(0) \dot{x}(0) - \dot{z}(0) \right) , \]
for any $t \in [0,T]$ because criticality on $[0,T]$ implies criticality on the subinterval $[0,t]$. It follows that $E = \der{\cL}{\dot{x}} \dot{x} - \dot{z}$ satisfies Equation \eqref{energy-diffeqn}.
\end{proof}

The usual argument that Lagrangian flows are symplectic, as presented for example in \cite[Section 1.2]{marsden2001discrete}, can be extended to show that flows of Herglotz' variational principle are contact transformations.

\begin{proposition}
Let $M = TQ \times \R$ with local coordinates $(x,\dot{x},z)$.
The flow $F:\R \times M \rightarrow M: (t,x,\dot{x},z) \mapsto F^t(x,\dot{x},z)$ of the generalized Euler-Lagrange equations consists of contact transformations $F^t$ with respect to the 1-form 
\[ \d z - p \,\d x , \]
where $p = \der{\cL}{\dot{x}}$.
\end{proposition}
\begin{proof}
On solutions of the generalized Euler-Lagrange equations, the value of $z(t)$ is uniquely defined by the initial values $x(0)$, $\dot{x}(0)$, and $z(0)$.
Any variation 
\[ v = (\delta x(0), \delta \dot{x}(0), \delta z(0)) \in T_{(x(0),\dot{x}(0),z(0))} M \]
of the initial data induces a variation 
\[ F^t_* v = (\delta x(t), \delta \dot{x}(t), \delta z(t)) \in T_{(x(t),\dot{x}(t),z(t))} M \]
at the endpoint, where $F^t_*: TM \rightarrow TM$ denotes the pushforward of $F^t$.

Since we are working on solutions of the generalized Euler-Lagrange equations, the integrand in Equation \eqref{general-variation} vanishes and only the boundary terms remain. They can be written as
\[ \d z (F^t_* v) 
= e^{B(t)} \left[ p(t) e^{-B(t)} \, \d x (F^t_* v) -  p(0) \, \d x(v) + \d z(v) \right], \]
where $p = \der{\cL}{\dot{x}}$. It follows that
\[ \left(F^t\right)^* (\d z - p \, \d x) = e^{B(t)} ( \d z - p \, \d x ),
\]
where $\left(F^t\right)^* : T^*M \rightarrow T^*M$ denotes the pullback of $F^t$.
Hence the flow consists of contact transformations with respect to the 1-form
$\d z - p \, \d x$
with conformal factor 
\begin{equation}\label{conformalfactor}
\mathrm{exp}(B(t)) = 
\mathrm{exp}\left(\int_0^t \der{\cL}{z}(\tau) \,\d \tau\right).
\end{equation}
\end{proof}

To close this section, let us briefly state the link of the Herglotz variational principle to the more common Hamiltonian formulation of contact dynamics. The contact Hamiltonian $H: T^* Q\times\R \rightarrow \R$ is defined by Legendre transformation
\[ H(x,p,z) = p \dot{x} - \cL(x,\dot{x},z) , \]
where $\dot{x}$ is eliminated from the right hand side by $p = \der{\cL}{\dot{x}}$. Taking the partial derivative with respect to $z$ gives
\[ \der{H}{z} = - \der{\cL}{z} , \]
hence from Equation \eqref{energy-diffeqn} it follows that
\[ \dot{H} = -\der{H}{z} H. \]
Differentiating instead with respect to $p$ and $x$ gives us the contact Hamiltonian equations:
\begin{align*}
    \der{H}{p} &= \dot{x} , \\
    \der{H}{x} &= - \der{\cL}{x} 
    = - \frac{\d}{\d t} \der{\cL}{\dot{x}} + \der{\cL}{z}\der{\cL}{\dot{x}} \\
    &= -\dot{p} - p \der{H}{z} .
\end{align*}

\section{Discrete Herglotz variational principle}\label{sec:dhvp}
As is standard in discrete mechanics, in what follows we replace $TQ$ by $Q^2$ (see e.g.~\cite{marsden2001discrete}).

\begin{definition}
Let $L: Q^2 \times \R^2 \times \R \rightarrow \R$ and $h > 0$. Given a discrete curve $x = (x_0,\ldots,x_N) \in Q^{N+1}$, we define $z = (z_0,\ldots,z_N) \in \R^{N+1}$ by $z_0 = 0$ and
\begin{equation}\label{z-difference}
 z_{j+1} - z_j = h L(x_j,x_{j+1},z_j,z_{j+1};h).
\end{equation}
The discrete curve $x$ is a \emph{solution to the discrete Herglotz variational principle} if
\begin{equation}\label{dVP-condition}
\der{z_{j+1}}{x_j} = 0,
\end{equation}
for all $j \in \{ 1,\ldots,N-1\}$.
\end{definition}

Note that Equation \eqref{z-difference} is a discrete version of Equation \eqref{z-differential}, and that Equation \eqref{dVP-condition} means that for a critical discrete curve $x$, a variation of $x_k$ can affect $z_k$ but none of the other $z_j$. In particular, this implies that $z_N$ is critical with respect to variations of $x_1,\ldots,x_{N-1}$. Most of the time we will consider a fixed step size $h$ and omit it from the notation of the discrete Lagrangian $L(x_j,x_{j+1},z_j,z_{j+1})$.

\begin{theorem}\label{thm-dgEL}
For a sufficiently small step size $h$, the discrete curve $x$ is a solution of the discrete Herglotz variational principle, with $z$ defined by Equation \eqref{z-difference}, {if and only if} it satisfies the \emph{discrete generalized Euler-Lagrange equations}
\begin{equation}\label{dgEL}
\D_1 L(x_j,x_{j+1},z_j,z_{j+1}) + \D_2 L(x_{j-1},x_j,z_{j-1},z_j) \frac{1 + h \D_3 L(x_j,x_{j+1},z_j,z_{j+1})}{1 - h \D_4 L(x_{j-1},x_j,z_{j-1},z_j)} = 0,
\end{equation} 
where $\D_i$ denotes the partial derivative with respect to the $i$-th entry.
\end{theorem}
Note that while in general the $x_j$ have several components, the $z_j$ are always scalar, hence $\D_1 L$ and $\D_2 L$ are vectors but $\D_3 L$ and $\D_4 L$ are scalars.

Equation \eqref{dgEL} is equivalent to
\begin{equation}\label{dgEL-alternative}
\begin{split}
& 0 = \D_2 L(x_{j-1},x_j,z_{j-1},z_j) + \D_1 L(x_j,x_{j+1},z_j,z_{j+1}) \\ 
&+ \frac{h \D_2 L(x_{j-1},x_j,z_{j-1},z_j)}{1 - h \D_4 L(x_{j-1},x_j,z_{j-1},z_j)} ( \D_3 L(x_j,x_{j+1},z_j,z_{j+1}) + \D_4 L(x_{j-1},x_j,z_{j-1},z_j) ). 
\end{split}
\end{equation}
In the first line one recognizes the usual discrete Euler-Lagrange equations. The term in the second line is a discretization of $\der{L}{\dot{x}} \der{L}{z}$.

\begin{proof}[Proof of Theorem \ref{thm-dgEL}]
From Equation \eqref{z-difference} it follows that
\begin{align*}
\der{z_{j+1}}{x_j}
= \der{z_j}{x_j} + h \D_1 L(x_j,x_{j+1},z_j,z_{j+1}) 
&+ h \D_3 L(x_j,x_{j+1},z_j,z_{j+1}) \der{z_j}{x_j} \\
&+ h \D_4 L(x_j,x_{j+1},z_j,z_{j+1}) \der{z_{j+1}}{x_j} .
\end{align*}
On solutions we have 
\begin{align*}
\der{z_{j+1}}{x_j} &= 0 , \\
\der{z_j}{x_j} &= h \D_2 L(x_{j-1},x_j,z_{j-1},z_j) +  h \D_4 L(x_{j-1},x_j,z_{j-1},z_j) \der{z_j}{x_j} ,
\end{align*}
where the derivative $\D_3 L$ is omitted because $\der{z_{j-1}}{x_j} = 0$.
It follows that
\begin{align*}
&\left( 1 -h \D_4 L(x_j,x_{j+1},z_j,z_{j+1}) \right) \der{z_{j+1}}{x_j} \\
&= h \D_1 L(x_j,x_{j+1},z_j,z_{j+1}) \\
&\quad + ( 1 + h \D_3 L(x_j,x_{j+1},z_j,z_{j+1}) ) \frac{h \D_2 L(x_{j-1},x_j,z_{j-1},z_j)}{1 - h \D_4 L(x_{j-1},x_j,z_{j-1},z_j)},
\end{align*} 
hence for sufficiently small $h$, $\der{z_{j+1}}{x_j} = 0$ is equivalent to Equation \eqref{dgEL}.
\end{proof}

In analogy to the continuous case, we will always assume the non-degeneracy condition 
\[ \left| \D_1 \D_2 L(x_j,x_{j+1},z_j,z_{j+1}) \right| \neq 0, \]
which guarantees that for sufficiently small $h$, Equation \eqref{dgEL} can be solved for $x_{j+1}$.

\begin{theorem}\label{thm-contact}
The map $Q^2 \times \R \mapsto Q^2 \times \R: (x_{j-1},x_j,z_{j-1}) \mapsto (x_j,x_{j+1},z_j),$
given by the generalized discrete Euler-Lagrange equations, induces a map 
\[ F:T^* Q \times \R \mapsto T^* Q \times \R: (x_{j-1},p_{j-1},z_{j-1}) \mapsto (x_j,p_j,z_j), \]
where
\begin{equation}\label{equalmomenta}
    p_j=p_j^-=p_j^+
\end{equation}
and
\begin{align}\label{pj-}
    p_j^- &= \frac{h \D_2 L(x_{j-1},x_j,z_{j-1},z_j)}{1 - h \D_4 L(x_{j-1},x_j,z_{j-1},z_j)} , \\
    \label{pj+}
    p_j^+ &= 
    -\frac{h \D_1L(x_{j},x_{j+1},z_j,z_{j+1})}{1+h\D_3L(x_{j},x_{j+1},z_j,z_{j+1})} .
\end{align} 
The map $F$ is a contact transformation with respect to the 1-form $\d z - p \,\d x$.
\end{theorem}

Equations \eqref{pj-} and \eqref{pj+} define the discrete Legendre transforms for contact systems, compare \cite[Section 1.5]{marsden2001discrete}.

\begin{proof} 
First note that the second equality in Equation \eqref{equalmomenta} follows from Equation \eqref{dgEL} and the definitions \eqref{pj-} and \eqref{pj+}.

To prove that $F$ is a contact transformation, we consider the case $j=2$.
The general statement is obtained by shifting all indices by the same integer. 

From
\[ z_2 - z_1 = h L(x_1,x_2,z_1,z_2) \]
it follows that
\begin{align*}
\d z_2 - \d z_1 
&= h \D_1 L(x_1,x_2,z_1,z_2) \,\d x_1 +  h \D_2 L(x_1,x_2,z_1,z_2) \,\d x_2 \\
&\quad + h \D_3 L(x_1,x_2,z_1,z_2) \,\d z_1 + h \D_4 L(x_1,x_2,z_1,z_2) \,\d z_2,
\end{align*}
hence, on solutions of the generalized Euler-Lagrange equations,
\begin{align*}
& (1 - h \D_4 L(x_1,x_2,z_1,z_2)) \,\d z_2 -  h \D_2 L(x_1,x_2,z_1,z_2) \,\d x_2 \\
&= \left( 1 +  h \D_3 L(x_1,x_2,z_1,z_2) \right) \d z_1 +  h \D_1 L(x_1,x_2,z_1,z_2) \,\d x_1 \\
&=  \left( 1 +  h \D_3 L(x_1,x_2,z_1,z_2) \right) \d z_1
- h \D_2 L(x_0,x_1,z_0,z_1) \frac{1 + h \D_3 L(x_1,x_2,z_1,z_2)}{1 - h \D_4 L(x_0,x_1,z_0,z_1)} \,\d x_1 .
\end{align*}
It follows that
\begin{align*}
&\d z_2 - \frac{h \D_2 L(x_1,x_2,z_1,z_2)}{1 - h \D_4 L(x_1,x_2,z_1,z_2)} \,\d x_2  \\
&\quad =  \frac{1 +  h \D_3 L(x_1,x_2,z_1,z_2)}{1 - h \D_4 L(x_1,x_2,z_1,z_2)}  \left(  \d z_1 - \frac{h \D_2 L(x_0,x_1,z_0,z_1)}{1 - h \D_4 L(x_0,x_1,z_0,z_1) } \,\d x_1 \right) . \qedhere
\end{align*}
\end{proof}

Note that the conformal factor
\[ \frac{1 + h \D_3 L}{1 - h \D_4 L} 
= 1 + h(\D_3 L + \D_4 L) + \cO(h^2)
= e^{h(\D_3 L + \D_4 L) } + \cO(h^2)\]
 is consistent with the continuous $\exp\!\left(\int_0^h \der{\cL}{z} \d t \right)$, cf.~Equation~\eqref{conformalfactor}.
We stress that Theorems~\ref{thm-dgEL} and~\ref{thm-contact} also apply to the case where $L$ does not depend on $z_2$, i.e.~$\D_4 L=0$.

A natural question to ask at this point is whether every contact transformation comes from a variational principle. Just like the symplectic counterpart to this question, it is answered in the affirmative using generating functions. 

Remarkably, the following result is stronger than the literal inverse to Theorem \ref{thm-contact}, which said that any discrete Lagrangian $L(x_1,x_2,z_1,z_2)$ yields a contact transformation. We will show that every contact transformation can be obtained from a discrete Lagrangian $L(x_1,x_2,z_1)$ that does not depend on the second instance of $z$.

We stress the importance of this result, since it implies that every contact integrator is variational. 

\begin{theorem}\label{thm-inverse}
Iterations of any contact transformation $(x_0,p_0,z_0) \mapsto (x_1,p_1,z_1)$ yield a discrete curve $x = (x_0,\ldots,x_N)$ that solves the discrete Herglotz variational principle for some discrete Lagrangian $L(x_j,x_{j+1},z_j)$.
\end{theorem}

Note that $L$ does not depend on $z_{j+1}$ in the statement of Theorem \ref{thm-inverse}. Hence without loss of generality we can restrict our attention to Lagrangians depending only on the first of the $z$-coordinates, as we will do e.g.~in Example~\ref{ex-harmonic}.

\begin{proof}[Proof of Theorem \ref{thm-inverse}]
As pointed out in \cite{bravetti2017contact}, the coordinate $z_1$ of a contact transformation $(x_0,p_0,z_0) \mapsto (x_1,p_1,z_1)$ can be considered as a generating function. We have
\[ \d z_1 - p_1 \,\d x_1 = f \left( \d z_0 - p_0 \,\d x_0 \right) . \]
Writing $z_1 = S(x_0,x_1,p_0,p_1,z_0)$ we find
\begin{align*}
    f \left( \d z_0 - p_0 \,\d x_0 \right) + p_1 \,\d x_1 &= \der{S}{x_0} \,\d x_0 + \der{S}{p_0} \,\d p_0 + \der{S}{z_0} \,\d z_0 + \der{S}{x_1} \,\d x_1 +  \der{S}{p_1} \,\d p_1 .
\end{align*}
It follows that $\der{S}{p_0} = \der{S}{p_1} = 0$ and
\begin{equation}\label{generating-function}
\begin{dcases}
    f = \der{S}{z_0} ,\\
    p_0 = -\left(\der{S}{z_0}\right)^{-1} \der{S}{x_0} ,\\
    p_1 = \der{S}{x_1} .
\end{dcases}
\end{equation}
Note that $S$ does not depend on $p_0$ or $p_1$, hence from now on we will write $S(x_0,x_1, z_0)$. Setting 
\[ L(x_0,x_1,z_0) = \frac{1}{h}\left( S(x_0,x_1, z_0) - z_0 \right), \]
iterations of the contact map satisfy
\[ z_{j+1} - z_j = h L(x_j,x_{j+1},z_j) . \]
Furthermore, using Equation \eqref{generating-function} we calculate that
\begin{align*}
    \der{z_{j+1}}{x_j} &= \der{S(x_j,x_{j+1},z_j)}{x_j} + \der{S(x_j,x_{j+1},z_j)}{z_j} \der{z_j}{x_j} \\
    &= \der{S(x_j,x_{j+1},z_j)}{x_j} + \der{S(x_j,x_{j+1},z_j)}{z_j} \der{S(x_{j-1},x_j,z_{j-1})}{x_j} \\
    &= - p_j \der{S(x_j,x_{j+1},z_j)}{z_j} + \der{S(x_j,x_{j+1},z_j)}{z_j} p_j \\
    &= 0,
\end{align*}
so the discrete curve $x$ obtained by iteration of the contact map satisfies the discrete Herglotz variational principle for $L$.
\end{proof}

\begin{example}\label{ex-harmonic}
The Lagrangian $\cL = \frac{1}{2} \dot{x}^2 - V(x) - \alpha z$ describes a mechanical system with Rayleigh dissipation (i.e.\@ a friction force linear in the velocity). The generalized Euler-Lagrange equation is
\[ \ddot{x} = - V'(x) - \alpha \dot{x} . \]
Note that $x$ need not be a scalar: the Lagrangian $\cL = \frac{1}{2} |\dot{x}|^2 - V(x) - \alpha z$ yields the analogous multi-component equation. This contrasts many other variational descriptions of the damped harmonic oscillator, which only apply to the scalar case \cite{musielak2008standard,cieslinski2010direct}. The same comment applies to the following discretization, which we write down for scalar $x$ but can easily be adapted to higher dimensions.

A discretization of the Lagrangian is
\begin{equation}\label{ex1-Lag-V}
L(x_j, x_{j+1}, z_j, z_{j+1}) = \frac{1}{2} \left( \frac{x_{j+1} - x_j}{h}\right)^2 - \frac{V(x_j) + V(x_{j+1})}{2} - \alpha z_j .
\end{equation}
Note that this Lagrangian depends only on $z_j$, not on $z_{j+1}$. Its discrete generalized Euler-Lagrange equations read
\begin{equation}\label{ex1-dEL-V}
\frac{x_{j+1} - 2 x_j + x_{j-1} }{h^2} = - V'(x_j) - \alpha \left( \frac{x_j - x_{j-1}}{h} - \frac{h}{2} V'(x_j) \right) .
\end{equation}
The discrete momentum can be calculated as
\[ p_j = \frac{h \D_2 L(x_{j-1},x_j,z_{j-1},z_j)}{1 - h \D_4 L(x_{j-1},x_j,z_{j-1},z_j)} = \frac{x_j - x_{j-1}}{h} - \frac{h}{2} V'(x_j) \]
or
\[ p_{j-1} = \frac{-h \D_1 L(x_{j-1},x_j,z_{j-1},z_j)}{1 + h \D_3 L(x_{j-1},x_j,z_{j-1},z_j)}
= \frac{ \frac{x_j - x_{j-1}}{h} + \frac{h}{2} V'(x_{j-1}) }{ 1 - h \alpha } . \]
We can implement the integrator explicitly in position-momentum formulation as
\begin{align*}
x_j &= x_{j-1} + h(1 - h \alpha) p_{j-1} - \frac{h^2}{2} V'(x_{j-1}) , \\
p_j &= (  1 - h \alpha) p_{j-1} - \frac{h}{2} \left( V'(x_j) + V'(x_{j-1}) \right).
\end{align*}

Let us consider the damped harmonic oscillator, $V(x) = \frac{1}{2} x^2$. Its equation of motion is
\[ \ddot{x} = - x - \alpha \dot{x} . \]
The above discrete Lagrangian then becomes
\begin{equation}\label{ex1-Lag}
L(x_j, x_{j+1}, z_j, z_{j+1}) = \frac{1}{2} \left( \frac{x_{j+1} - x_j}{h}\right)^2 - \frac{1}{4} \left( x_j^2 + x_{j+1}^2 \right) - \alpha z_j
\end{equation}
and it discrete generalized Euler-Lagrange equations read
\begin{equation}\label{ex1-dEL}
\frac{x_{j+1} - 2 x_j + x_{j-1} }{h^2} = - x_j - \alpha \left( \frac{x_j - x_{j-1}}{h} - \frac{h}{2} x_j \right) .
\end{equation}
The position-momentum formulation of the integrator gives
\begin{align*}
x_j &= \left( 1 - \frac{h^2}{2} \right) x_{j-1} + h(1 - h \alpha) p_{j-1} , \\
p_j &= (  1 - h \alpha) p_{j-1} - \frac{h}{2} ( x_j + x_{j-1} ).
\end{align*}
\end{example}

\begin{example}\label{ex-harmonic-symmetric}
	For the theory of discrete contact systems by itself, it is sufficient to have the Lagrangian depend on $z_j$ but not on $z_{j+1}$, as we saw in Theorem \ref{thm-inverse}.	For the sake of a good numerical approximation, however, it is beneficial to relax this condition. Continuing the example of a damped mechanical system, we can take the discrete Lagrangian
	\begin{equation}\label{ex2-Lag-V}
	L(x_j, x_{j+1}, z_j, z_{j+1}) = \frac{1}{2} \left( \frac{x_{j+1} - x_j}{h}\right)^2 - \frac{V(x_j) + V(x_{j+1})}{2} - \alpha \frac{z_j + z_{j+1}}{2} .
	\end{equation}
	Note the difference with Example \ref{ex-harmonic}: now $L$ depends also on $z_{j+1}$.	Its discrete generalized Euler-Lagrange equations read
	\begin{equation}\label{ex2-dEL-V}
	\frac{x_{j+1} - 2 x_j + x_{j-1} }{h^2} = - V'(x_j) - \frac{\alpha}{1 + \frac{h}{2} \alpha} \left( \frac{x_j - x_{j-1}}{h} - \frac{h}{2} V'(x_j) \right) .
	\end{equation}
	Equations \eqref{ex1-dEL-V} and  \eqref{ex2-dEL-V} are related by a simple change in the parameter $\alpha$. This minor difference should not be dismissed, though, as the discrete Lagrangian \eqref{ex2-Lag-V} is a second order approximation of the continuous Lagrangian, compared to the first order approximation of Equation \eqref{ex1-Lag-V}. What we mean by this will be clarified in the next section: see Example \ref{ex:ex1cont} and Example \ref{ex:ex2cont}.
	
	The discrete momentum for the Lagrangian \eqref{ex2-Lag-V} can be calculated as
	\[ p_j^- = \frac{h \D_2 L(x_{j-1},x_j,z_{j-1},z_j)}{1 - h \D_4 L(x_{j-1},x_j,z_{j-1},z_j)} = \frac{ \frac{x_j - x_{j-1}}{h} - \frac{h}{2} V'(x_j) }{1 + \frac{h}{2}\alpha }\]
	or
	\[ p_{j-1}^+ = \frac{-h \D_1 L(x_{j-1},x_j,z_{j-1},z_j)}{1 + h \D_3 L(x_{j-1},x_j,z_{j-1},z_j)}
	= \frac{ \frac{x_j - x_{j-1}}{h} + \frac{h}{2} V'(x_{j-1}) }{ 1 - \frac{h}{2} \alpha } . \]
	The generalized Euler-Lagrange equations state that both formulas for the discrete momentum agree. 
	On solutions, we have that
	\[ p_j = p_j^+ = p_j^- = \frac{x_{j+1} - x_{j-1}}{2h}. \]
	
	We can implement the integrator explicitly in position-momentum formulation as
	\begin{align*}
	x_j &= x_{j-1} + h \left( 1 - \frac{h}{2} \alpha \right) p_{j-1} - \frac{h^2}{2} V'(x_{j-1}) , \\
	p_j &= \frac{ \left( 1 - \frac{h}{2} \alpha \right) p_{j-1} - \frac{h}{2} \left( V'(x_j) + V'(x_{j-1})  \right) }{ 1 + \frac{h}{2}\alpha } .
	\end{align*}
	
	Equation \eqref{ex2-dEL-V} is also the second order difference equation corresponding to the leapfrog (St\"ormer-Verlet) method
	\begin{align*}
	    x_{j+1} &= x_j + h \pi_{j + \frac{1}{2}} \\
	    \pi_{j + \frac{1}{2}} &= \pi_{j - \frac{1}{2}} - h \left( V'(x_j) + \frac{\alpha }{2} \left( \pi_{j + \frac{1}{2}} + \pi_{j - \frac{1}{2}} \right) \right).
	\end{align*}
	Indeed, eliminating the momentum $\pi$ from this system we find
	\begin{align*}
	    x_{j+1} - 2 x_j + x_{j-1}
	    &= h \left( \pi_{j + \frac{1}{2}} - \pi_{j - \frac{1}{2}} \right) \\
	    &= -h^2 V'(x_j) - \frac{h^2 \alpha}{2} \left( \pi_{j + \frac{1}{2}} + \pi_{j - \frac{1}{2}} \right) \\
	    &= -h^2 V'(x_j) - \frac{h \alpha}{2} (x_{j+1} - x_{j-1}) \\
	    &= -h^2 V'(x_j) - \frac{h \alpha}{2} (x_{j+1} - 2 x_j + x_{j-1}) - h \alpha (x_j - x_{j-1}) 
	\end{align*}
	hence
	\[ \left( 1 + \frac{h \alpha}{2} \right) (x_{j+1} - 2 x_j + x_{j-1}) = -h^2 V'(x_j) - h \alpha (x_j - x_{j-1}), \]
	which is equivalent to Equation \eqref{ex2-dEL-V}.
	
	The momenta at integer steps can be included in the leapfrog method by adding one internal stage:
	\begin{align*}
	    \pi_{j + \frac{1}{2}} &= \pi_{j} - \frac{h}{2} \left( V'(x_j) + \alpha \pi_{j + \frac{1}{2}} \right) \\
	    x_{j+1} &= x_j + h \pi_{j + \frac{1}{2}} \\
	    \pi_{j + 1} &= \pi_{j + \frac{1}{2}} - \frac{h}{2} \left( V'(x_{j+1}) + \alpha  \pi_{j + \frac{1}{2}} \right).
	\end{align*}
	We find that
	\begin{align*}
	 \pi_j &= \frac{\pi_{j + \frac{1}{2}} + \pi_{j - \frac{1}{2}}}{2} +\frac{h}{4} \alpha \left(\pi_{j + \frac{1}{2}} - \pi_{j - \frac{1}{2}} \right)\\
	&=\frac{x_{j+1} - x_{j-1}}{2h} + \frac{h}{4} \alpha \left(\pi_{j + \frac{1}{2}} - \pi_{j - \frac{1}{2}} \right) 
	=\left(1 + \frac{h^2}{4} \alpha^2 \right) p_j  + \frac{h^2}{4} \alpha V'(x_j), 
	\end{align*}
	hence when initialized with the same momentum, the difference between the result of our contact method and the leapfrog solution will be of order $h^2 (\alpha + \alpha^2)$. Whether $p_0$ or $\pi_0$ is a better approximation for the true initial momentum $\dot{x}$ depends on the initial conditions.
\end{example}

\begin{example}
For a more general contact system, motivated by \cite{sloan2018dynamical}, consider the Lagrangian
\[\cL = \frac{1}{2} \dot{x}^2 - V(x) - \frac{1}{2} \alpha z^2 .\]
The equations of motion are
\[ \begin{dcases}
     \ddot{x} = - V'(x) - \alpha z \dot{x} \\
     \dot{z} = \frac{1}{2} \dot{x}^2 - V(x) - \frac{1}{2} \alpha z^2 ,
\end{dcases}  \]
or, in position-momentum formulation,
\[ \begin{dcases}
     \dot{x} = p \\
     \dot{p} = - V'(x) - \alpha z \dot{x} \\
     \dot{z} = \frac{1}{2} \dot{x}^2 - V(x) - \frac{1}{2} \alpha z^2 .
\end{dcases} \]
Note that the Euler-Lagrange equation explicitly involves $z$ in this case, so it is not possible to solve the equations for $x$ and $p$ separately. This means that symplectic integrators cannot be applied (unless one adds a dummy variable to the systems in order to obtain an even-dimensional system once again). In comparison, in Examples \ref{ex-harmonic} and \ref{ex-harmonic-symmetric} a symplectic integrator could be applied, but it would not respect the contact structure. 

Consider the discrete Lagrangian
\[ L(x_j, x_{j+1}, z_j, z_{j+1}) = \frac{1}{2} \left( \frac{x_{j+1} - x_j}{h}\right)^2 - \frac{V(x_j) + V(x_{j+1})}{2} - \frac{1}{4} \alpha z_j^2 - \frac{1}{4} \alpha z_{j+1}^2 .\]
The discrete momenta are
\[ p_j = \frac{h \D_2 L(x_{j-1},x_j,z_{j-1},z_j)}{ 1 - h \D_4 L(x_{j-1},x_j,z_{j-1},z_j) }  = \frac{ \frac{x_j - x_{j-1}}{h} - \frac{h}{2} V'(x_j) }{ 1 + \frac{h}{2}\alpha z_j }\]
and
\[ p_{j-1} = \frac{-h \D_1 L(x_{j-1},x_j,z_{j-1},z_j)}{1 + h \D_3 L(x_{j-1},x_j,z_{j-1},z_j)}
= \frac{ \frac{x_j - x_{j-1}}{h} + \frac{h}{2} V'(x_{j-1}) }{ 1 - \frac{h}{2} \alpha z_{j-1} } . \]
Hence we find an implicit contact integrator
\[ \begin{dcases}
    x_{j+1} = x_j + h \left( 1 - \frac{h}{2} \alpha z_j \right) p_j - \frac{h^2}{2} V'(x_j) \\
    p_{j+1} = \frac{ \left( 1 - \frac{h}{2} \alpha z_j \right) p_j  - \frac{h}{2}\left( V'(x_j) + V'(x_{j+1}) \right) }{ 1 + \frac{h}{2} \alpha z_{j+1}} \\
    z_{j+1} = z_j + h L(x_j, x_{j+1}, z_j, z_{j+1}) .
\end{dcases}\]

\end{example}

\section{Backward error analysis}\label{sec:bea}
A central idea to explain the long-time behavior of symplectic integrators is the study of modified differential equations whose solutions interpolate the discrete solutions of a discrete system of equations.
This idea, looking for a perturbed continuous system that exactly corresponds to the discretization, is an example of backward error analysis.
It is a well-known and essential fact that if a symplectic integrator is applied to a Hamiltonian equation, then the resulting modified equation is Hamiltonian as well.
Similarly, when a classical variational integrator is applied to a Lagrangian system, the resulting modified equation is Lagrangian \cite{vermeeren2017modified}. 
Below we establish that an analogous result holds for contact variational integrators.

First let us have a look at the general form of the discrete generalized Euler-Lagrange equations.

\begin{proposition}\label{prop-consistent}
Consider a continuous non-degenerate Lagrangian $\cL(x,\dot{x},z)$ with generalized Euler-Lagrange equation
$\ddot{x} = f(x, \dot{x}, z)$ and a consistent discretization $L(x_j,x_{j+1},z_j,z_{j+1};h)$ of $\cL$, by which we mean that for any smooth $x$ and $z$ there holds
\[ L(x(t),x(t+h),z(t),z(t+h);h) = \cL(x(t),\dot{x}(t),z) + \cO(h). \]
Then the discrete generalized Euler-Lagrange equation is a consistent discretization of the continuous generalized Euler-Lagrange equation, i.e.\@ it takes the form
\begin{equation}\label{2nd-order}
\frac{x_{j-1} - 2 x_j + x_{j+1}}{h^2} = F(x_{j-1},x_j,x_{j+1},z_{j-1},z_j,z_{j+1};h),
\end{equation}
where for any smooth $x$ and $z$
\[F(x(t-h),x(t),x(t+h),z(t-h),z(t),z(t+h);h) = f( x(t), \dot{x}(t), z) + \cO(h). \] 
\end{proposition}
We give a formal proof. A rigorous version of this argument is obtained by generalizing the corresponding proof in \cite{vermeeren2017modified} to the case of the Herglotz variational principle. A different proof strategy can be found in \cite[Section 2.3]{marsden2001discrete}.

\begin{proof}[Proof of Proposition \ref{prop-consistent}]
Let $x$ be a smooth curve interpolating solutions of the discrete generalized Euler-Lagrange equation \eqref{dgEL-alternative}. Then $\cE_1 + \cE_2 \cE_3 = 0$, where
\[ \cE_1 = \D_1 L(x(t),x(t+h),z(t),z(t+h);h) + \D_2 L(x(t-h),x(t),z(t-h),z(t);h) , \]
\[ \cE_2 = \frac{h \D_2 L(x(t-h),x(t),z(t-h),z(t);h)}{1 - h \D_4 L(x(t-h),x(t),z(t-h),z(t);h) } , \]
and
\[ \cE_3 = \D_3 L(x(t),x(t+h),z(t),z(t+h);h) + \D_4 L(x(t-h),x(t),z(t-h),z(t);h) . \]
We start by showing that $\cE_1$ is a consistent discretization of the classical Euler-Lagrange equation.
In terms of the Taylor expansions of the shifted variables, it becomes
\begin{align*}
    \cE_1 &= \left( \der{}{x} - \frac{1}{h} \der{}{\dot{x}} \right) L(x, x+h\dot{x}+\ldots,z,z+h\dot{z}+\ldots;h) \\
    &\quad + \left( \der{}{x} + \frac{1}{h} \der{}{\dot{x}} \right) L(x-h\dot{x}+\ldots,x,z-h\dot{z}+\ldots,z;h) + \cO(h) .
\end{align*}
Since the Lagrangian is assumed to be a consistent discretization, we have 
\[ L(x(t),x(t+h),z(t),z(t+h);h) = \cL(x(t),\dot{x}(t),z(t)) + \cO(h), \]
hence
\begin{align*}
    \cE_1 &= \left( \der{}{x} - \frac{1}{h} \der{}{\dot{x}} \right) \cL(x(t),\dot{x}(t),z(t)) \\
    &\quad + \left( \der{}{x} + \frac{1}{h} \der{}{\dot{x}} \right) \cL(x(t-h),\dot{x}(t-h),z(t-h)) + \cO(h) 
\end{align*}
In other words
\[ \cE_1 = \left( \der{}{x} - \frac{1}{h} \der{}{\dot{x}} \right) \cL + \left( \der{}{x} + \frac{1}{h} \der{}{\dot{x}} \right) \left( \cL - h \frac{\d\cL}{\d t} \right) + \cO(h), \]
where each $\cL$ is evaluated at $(x(t),\dot{x}(t),z(t))$. We can simplify this expression using the fact that
\[ \der{}{\dot{x}} \frac{\d\cL}{\d t} = \der{}{\dot{x}} \left( \der{\cL}{x} \dot{x} + \der{\cL}{\dot{x}} \ddot{x} + \der{\cL}{z} \dot{z} \right)
= \frac{\d}{\d t} \der{\cL}{\dot{x}} + \der{\cL}{x} \]
and obtain
\[ \cE_1 = \der{\cL}{x} - \frac{\d}{\d t} \der{\cL}{\dot{x}}  + \cO(h) . \]
A perfectly analogous computation yields that
\[ \cE_3 = \der{\cL}{z} - \frac{\d}{\d t} \der{\cL}{\dot{z}}  + \cO(h) =  \der{\cL}{z} + \cO(h) . \]
Finally, we compute
\begin{align*}
    \cE_2 &= \frac{ h \left( \der{}{x} + \frac{1}{h} \der{}{\dot{x}} \right) L(x-h\dot{x}+\ldots,x,z-h\dot{z}+\ldots,z;h) }{ 1 - h \left( \der{}{z} + \frac{1}{h} \der{}{\dot{z}} \right) L(x-h\dot{x}+\ldots,x,z-h\dot{z}+\ldots,z;h) } \\
    &= \frac{h \left( \der{}{x} + \frac{1}{h} \der{}{\dot{x}} \right) \cL(x(t-h),\dot{x}(t-h),z(t-h))}{ 1 - h \left( \der{}{z} + \frac{1}{h} \der{}{\dot{z}} \right) \cL(x(t-h),\dot{x}(t-h),z(t-h)) } \\
    &= \frac{ \der{}{\dot{x}} \left( \cL - h \frac{\d\cL}{\d t} \right) }{ 1 - \der{}{\dot{z}} \left( \cL - h \frac{\d\cL}{\d t} \right)} \\
    &= \der{\cL}{\dot{x}} + \cO(h) .
\end{align*}
It follows that
\[ \cE_1 + \cE_2 \cE_3 = \der{\cL}{x} - \frac{\d}{\d t} \der{\cL}{\dot{x}} + \der{\cL}{\dot{x}} \der{\cL}{z} + \cO(h) . \]
The claimed result follows by isolating the term $\ddot{x}$ in the right hand side of this equation and the term $\frac{x(t-h) - 2 x(t) + x(t+h)}{h^2}$ in the left hand side.
\end{proof}

Now we turn our attention to the modified equations of the system
\begin{equation}\label{disc-system}
\begin{dcases}
\frac{z_{j+1} - z_{j}}{h} = L(x_j,x_{j+1},z_j,z_{j+1};h) \\
\frac{x_{j+1} - 2 x_j + x_{j-1} }{h^2} = F(x_{j-1},x_j,x_{j+1},z_{j-1},z_j,z_{j+1};h) .
\end{dcases}
\end{equation}
That is, we look for differential equations whose solutions interpolate solutions of the difference equations. The precise definition of a modified equation is a bit more involved because of convergence issues.

\begin{definition}
The \emph{system of modified equations} for the difference system \eqref{disc-system} is defined by the formal expressions
\begin{equation}\label{mod-eqn}
\begin{dcases}
\dot{z} = \cL_{\mod}(x,\dot{x},z,h) = \cL(x,\dot{x},z) + h \cL_1(x,\dot{x},z) + h^2 \cL_2(x,\dot{x},z) + \ldots \\
\ddot{x} = f_{\mod}(x,\dot{x},z;h)  = f(x,\dot{x},z) + h f_1(x,\dot{x},z) + h^2 f_2(x,\dot{x},z) + \ldots
\end{dcases} 
\end{equation}
such that for any $k \in \N$, every solution $(x,z)$ of the truncated differential equations
\begin{equation}\label{mod-eqn-trunc}
\begin{dcases}
\dot{z} = \cL(x,\dot{x},z) + h \cL_1(x,\dot{x},z) + \ldots + h^k \cL_k(x,\dot{x},z)\\
\ddot{x} = f(x,\dot{x},z) + h f_1(x,\dot{x},z) + \ldots + h^k f_k(x,\dot{x},z)
\end{dcases}
\end{equation}
satisfies the difference equations with a defect of order $k+1$, in the sense that
\[ \begin{dcases}
\frac{z(t+h) - z(t)}{h} = L(x(t),x(t+h),z(t),z(t+h);h) + \cO(h^{k+1}) \\
\frac{x(t+h) - 2 x(t) + x(t-h) }{h^2} = F(x(t-h),x(t),x(t+h),z(t-h),z(t),z(t+h);h) \\
\hspace{10.7cm}+ \cO(h^{k+1}) .
\end{dcases} \]
\end{definition}
Given a difference equation of the form \eqref{2nd-order} one can recursively compute the coefficients $\cL_i$ and $f_i$ of the system of modified equations. Examples of such calculations can be found for example in \cite[Chapter IX]{hairer2006geometric} and \cite{vermeeren2017modified}. Note that in the leading order of the modified equations we recover the original differential equations. This is because we are dealing with consistent discretizations. The additional terms of the power series contain information about the integrator. In particular, the order of an integrator is the smallest $k > 0$ such that the $h^k$-term in Equation \eqref{mod-eqn} is non-zero.

If we look at the difference equation for $z$ by itself, i.e.\@ with an arbitrary smooth curve $x$ instead of one interpolating a discrete solution, we get a different modified equation, depending on higher derivatives of $x$:
\begin{equation}\label{mod-eqn-z}
\begin{split}
\dot{z} &= \cL_{\mod}^z(x,\dot{x},\ddot{x},\ldots,z,h) \\
&= \cL(x,\dot{x},z) + h \cL_1^z(x,\dot{x},\ddot{x},\ldots,z) + h^2 \cL_2^z(x,\dot{x},\ddot{x},\ldots,z) + \ldots
\end{split}
\end{equation}
such that for any $k \in \N$ and any smooth curve $x$, every solution $z$ of the truncated differential equation
\begin{equation}\label{mod-eqn-trunc-z}
\dot{z} = \cL(x,\dot{x},z) + h \cL_1^z(x,\dot{x},\ddot{x},\ldots,z) + \ldots + h^k \cL_k^z(x,\dot{x},\ddot{x},\ldots,z)
\end{equation}
satisfies 
\[ \frac{z(t+h) - z(t)}{h} = L(x(t),x(t+h),z(t),z(t+h);h) + \cO(h^{k+1}). \]

In the system of modified equations \eqref{mod-eqn}, $\cL_{\mod}(x,\dot{x},z,h)$ is the Herglotz Lagrangian for $\ddot{x} = f_{\mod}(x,\dot{x},z;h)$ in the following sense:
\begin{theorem}\label{thm-modified}
A truncation after order $h^k$ of the power series $\cL_{\mod}(x,\dot{x},z,h)$ yields as its generalized Euler-Lagrange equations $\ddot{x} = f_{\mod}(x,\dot{x},z;h) + \cO(h^{k+1})$.
\end{theorem}

\begin{proof}[Sketch of proof]
Let $(x(t),z(t))$ be any solution to the truncated system of modified equations \eqref{mod-eqn-trunc}. By definition of a modified equation, the discrete curve $(x_j,z_j)_{j\in \Z}$ defined by $x_j = x(jh)$ and $z_j = z(jh)$ satisfies the discrete system \eqref{disc-system} with a defect of order $\cO(h^{k+1})$. 

Now consider the action $z_N = \bar z(Nh)$ where $\bar z$ is a solution of the higher order Equation \eqref{mod-eqn-trunc-z}, with $x$ as before. The discrete Herglotz variational principle implies that $z_N$ is critical with respect to variations of $x(t)$, supported on the interval $(0,Nh)$, again up to a defect of order $\cO(h^{k+1})$. This means that $x(t)$ solves the continuous Herglotz problem for $\cL_{\mod}^z$ with the same defect. 

We want to prove the same property with $\cL_{\mod}$ instead of $\cL_{\mod}^z$. If we define $z$ by Equation \eqref{mod-eqn-trunc}, independent of second or higher derivatives of $x$, then it will only interpolate a discrete solution if $x$ solves the modified equation $\ddot{x} = f_{\mod}(x,\dot{x},z,h) + \cO(h^{k+1})$. Hence we do not have the freedom to take variations of $x$ as needed in the argument above.

To show that $\cL_{\mod}(x,\dot{x},z,h)$ is nevertheless a Lagrangian for the modified equation, we need to show that replacing higher derivatives of $x$ in $\cL_{\mod}^z(x,\dot{x},\ddot{x},\ldots,z,h)$ using the modified equation does not change its generalized Euler-Lagrange equations. Sufficient conditions for this are that 
\[ \frac{\cL_{\mod}^z}{x^{(\ell)}} =  \cO(h^{k+1}) \qquad \forall \ell \geq 2, \]
where $x^{(\ell)}$ denotes the $\ell$-th derivative of $x$. These conditions can be obtained from the so-called meshed variational problem, as explained in \cite{vermeeren2017modified}. Then it follows that $x(t)$ satisfies the Herglotz variational principle for the first order Lagrangian $\cL_{\mod}(x,\dot{x},z;h)$.
\end{proof}

\begin{example}[Example \ref{ex-harmonic} continued]\label{ex:ex1cont}
Let us calculate the first order approximation of the modified equations for our first discretization of the damped harmonic oscillator. Assume that $x$ is a solution to the modified equation. Then it satisfies Equation \eqref{ex1-dEL} when we replace $x_j$ by $x(t)$ and $x_{j \pm 1}$ by $x(t \pm h)$:
\[ \frac{x(t+h) - 2 x(t) + x(t-h) }{h^2} = - x(t) - \alpha \left( \frac{x(t) - x(t-h)}{h} - \frac{h}{2} x(t) \right) . \]
A Taylor expansion gives
\[ \ddot{x} = - x - \alpha \left( \dot{x} - \frac{h}{2} \ddot{x} - \frac{h}{2} x \right) + \cO(h^2), \]
where all instances of $x$ and its derivatives are evaluated at $t$. Since in the leading order of the modified equation we recover the original equation, we know that $\ddot{x} = - x - \alpha \dot{x} + \cO(h)$, which we can use to simplify the right hand side. We obtain 
\begin{equation}
     \ddot{x} = - x - \alpha \dot{x} - \frac{h \alpha^2}{2} \dot{x} + \cO(h^2) . \label{ex2-modEL}
\end{equation}

Using the same procedure we calculate the modified equation for $\frac{z_{j+1} - z_j}{h} = L$, with $L$ given by \eqref{ex1-Lag}. In terms of the interpolating functions, the difference equation reads
\[ \frac{z(t+h) - z(t)}{h} =  \frac{1}{2} \left( \frac{x(t+h) - x(t)}{h}\right)^2 - \frac{1}{4} \left( x(t)^2 + x(t+h)^2 \right) - \alpha z(t). \]
and its Taylor expansion is
\[ \dot{z} + \frac{h}{2} \ddot{z} = \frac{1}{2} \left( \dot{x} + \frac{h}{2} \ddot{x} \right)^2 - \frac{1}{4} x^2 - \frac{1}{4}\left( x + h \dot{x} \right)^2 - \alpha z + \cO(h^2). \]
Solving this for $\dot{z}$ we find
\[ \dot{z} = \frac{1}{2} \dot{x}^2 - \frac{1}{2} x^2 - \alpha z + \frac{h}{2}\left( \dot{x}\ddot{x} - x \dot{x} - \ddot{z} \right) + \cO(h^2). \]
In the right hand side we replace $\ddot{z}$ using the leading order equation
\[ \ddot{z} = \frac{\d}{\d t}\left( \frac{1}{2} \dot{x}^2 - \frac{1}{2} x^2 - \alpha z \right) + \cO(h) = \dot{x} \ddot{x} - x \dot{x} - \alpha \left( \frac{1}{2} \dot{x} - \frac{1}{2} x^2 - \alpha z \right) + \cO(h) \]
to find
\begin{equation} \label{ex2-modz}
\dot{z} = \frac{1}{2} \dot{x}^2 - \frac{1}{2} x^2 - \alpha z + \frac{h \alpha}{2}\left( \frac{1}{2} \dot{x} - \frac{1}{2} x^2 - \alpha z \right) + \cO(h^2).
\end{equation}
The right hand side of this modified equation should give us the modified Lagrangian. A simple calculation shows that the generalized Euler-Lagrange equation for 
\[ \cL_{\mod} = \frac{1}{2} \dot{x}^2 - \frac{1}{2} x^2 - \alpha z + \frac{h\alpha}{2} \left( \frac{1}{2} \dot{x}^2 - \frac{1}{2} x^2 - \alpha z \right) + \cO(h^2) \]
is indeed Equation \eqref{ex2-modEL}.
Note that up to the $\cO(h^2)$ error term, the modified Lagrangian $\cL_\mod$ is a rescaling of the original Lagrangian $\cL$. Unlike for the classical variational principle, this does not imply that both Lagrangians have the same generalized Euler-Lagrange equations.

The second order term of the modified equations can be calculated by including one more term in the Taylor expansions and simplifying the right hand sides using the first order modified equations \eqref{ex2-modEL}--\eqref{ex2-modz} instead of the leading order equations. We find
\begin{equation}
    \ddot{x} = - x - \alpha \dot{x} - \frac{h \alpha^2}{2} \dot{x} - \frac{h^2}{12} \left( \left( \alpha^2 + 1 \right) x + 4 \alpha^3 \dot{x} \right) + \cO(h^3)
\end{equation}
and
\begin{equation}
\begin{split}
\dot{z} &= \frac{1}{2} \dot{x}^2 - \frac{1}{2} x^2 - \alpha z + \frac{h \alpha}{2}\left( \frac{1}{2} \dot{x} - \frac{1}{2} x^2 - \alpha z \right) \\
&\quad + \frac{h^2}{24} \left( (4 \alpha^2 - 1) x^2 - (5 \alpha^2 - 2) \dot{x}^2 - 4 \alpha x \dot{x} + 8 \alpha^3 z \right) + \cO(h^3)
\end{split}
\end{equation}
This process can be continued to recursively find the modified equations to any order.
\end{example}

\begin{example}[Example \ref{ex-harmonic-symmetric}  continued]\label{ex:ex2cont}
We can repeat the above calculation to obtain a modified equation for the second order integrator too. We find
\begin{align*}
    \ddot{x} &= -x - \alpha \dot{x} - \frac{h^2}{12} \left( \alpha^3 \dot{x} + \alpha^2 x + x \right) + \cO(h^3) , \\
    \dot{z} &=  \frac{1}{2} \dot{x}^2 - \frac{1}{2} x^2 - \alpha z + \frac{h^2}{24} \left( (\alpha^2 - 1) x^2 - (2 \alpha^2 - 2) \dot{x}^2 - 4 \alpha x \dot{x} + 2 \alpha^3 z \right) + \cO(h^3) ,
\end{align*}
which shows that the discretization of Example \ref{ex-harmonic-symmetric} is indeed a second order method.
This is a consequence of the symmetry of the discretization:
\begin{align*}
    \frac{V(x(t)) + V(x(t+h))}{2} &= V \!\left(x \!\left(t+\frac{h}{2} \right)\right) + \cO(h^2) , \\
    \frac{x(t+h) - x(t)}{h} &= \dot{x} \!\left(t+\frac{h}{2} \right) + \cO(h^2) , \\
    \frac{z(t) + z(t+h)}{2} &= z \!\left(t+\frac{h}{2} \right) + \cO(h^2) .
\end{align*}
\end{example}

\section{On the explicit time dependence}\label{sec:tdep}

Even though in the previous sections we focused on Lagrangians that do not explicitly depend on time, going through the previous proofs and examples one can observe that there is no obstruction to considering explicitly time-dependent systems. In fact, when the system depends explicitly on time, the resulting flow yields a time-dependent contact transformation, in compete analogy to what happens with time-dependent canonical transformations in the symplectic case. We refer to~\cite{bravetti2017contact} for more on the theory of time-dependent contact transformations and their generating functions.
Therefore, with some efforts and modulo a slight complication of the notation in few instances, it is possible to extend all the previous results to such systems in a straightforward way, so that the resulting maps are discretizations of the corresponding time-dependent contact transformations.

How the explicitly time-dependent terms appear in the discrete Lagrangian and the resulting difference equation depend both on the choice of discretization and on the form of time-dependent terms in the continuous Lagrangian. In many cases, such as for external forcing, the time-dependence can be separated neatly and the final result will be elegant and readable.  To illustrate the time-dependent case, we build upon the Lagrangian presented in Example \ref{ex-harmonic} and Example \ref{ex-harmonic-symmetric}, and consider a forced damped harmonic oscillator. 

\begin{example}\label{ex-harmonic-time-dep}
In general, the Lagrangian of a mechanical system with Raileigh dissipation and external forcing $f(t)$ looks like
\[
\mathcal{L}(t, x, \dot{x}, z) = \frac12 \dot{x}^2 -V(x) -\alpha z + \boxed{f(t) x}.
\]
Indeed, the generalized Euler-Lagrange equation in this case is
\[
\ddot{x} = - V'(x) - \alpha \dot{x} + \boxed{f(t)}.
\]
Here and in what follows, we emphasize the difference with the equations from Example \ref{ex-harmonic-symmetric} in boxes.

Let $t_j = t_{j-1} + h = t_0 + jh$, then a natural discretization of the Lagrangian above is
\begin{equation}\label{Lag-forcing}
\begin{split}
L(t_j,t_{j+1},x_j, x_{j+1}, z_j, z_{j+1}) = & \frac{1}{2} \left( \frac{x_{j+1} - x_j}{h}\right)^2 - \frac{V(x_j) + V(x_{j+1})}{2} - \alpha \frac{z_j + z_{j+1}}2 \\
    & + \boxed{\frac{f(t_j) x_j + f(t_{j+1}) x_{j+1}}2}
\end{split}
\end{equation}
and the discrete generalized Euler-Lagrange equation reads
\begin{equation}\label{dEL-Forcing}
\begin{split}
\frac{x_{j+1} - 2 x_j + x_{j-1} }{h^2} &= - V'(x_j) + \boxed{f(t_j)} \\
&\quad - \frac{\alpha}{1 + \frac{h}2\alpha} \left( \frac{x_j - x_{j-1}}{h} - \frac{h}{2} V'(x_j) + \boxed{\frac{h}{2} f(t_j)}\right) .
\end{split}
\end{equation}
The position-momentum formulation of the integrator is
\begin{equation}
\begin{split}
x_j &= x_{j-1} + h\left(1 - \frac h2 \alpha\right) p_{j-1} - \frac{h^2}{2} V'(x_{j-1})
     + \boxed{\frac{h^2}2 f(t_{j-1})}, \\
p_j &= \frac{
\left(1 - \frac{h}{2} \alpha\right) p_{j-1} - \frac{h}{2} ( V'(x_j) + V'(x_{j-1}) )
     + \boxed{\textstyle \frac{h}{2} \big(f(t_j) + f(t_{j-1})\big)}
     }{1 + \frac h2\alpha}.
\end{split}
\end{equation}
\end{example}

\section{Numerical results}\label{sec:num}

In this section we discuss the behaviour of our contact variational integrators in comparison with some classical fixed step methods.
In what follows we consider the damped harmonic oscillator with and without forcing, integrated using:
\begin{itemize}
    \item our contact variational integrators of both first and second order as presented in the previous examples Example \ref{ex-harmonic} and Example \ref{ex-harmonic-symmetric} (respectively denoted ``Contact (1st)'' and ``Contact (2nd)''),
    \item the symplectic second order Leapfrog (also known as St\"ormer-Verlet) \cite{hairer2003geometricnumerical, hairer2006geometric},
    \item the third order Ruth3 integrator \cite{CandyRozmus1991,ruth1983},
    \item a second order variational but non-contact (VNC) method for forced Lagrangian systems \cite{dediego2018variational} obtained by a Verlet discretization of a Lagrangian in duplicated phase space \cite{galley2013classical}, 
    \[ \frac{x_{j+2} - 2 x_{j+1} + x_{j}}{h^2} + \alpha \frac{x_{j+2} - x_{j}}{2h} + V'(x_{j_1})  = 0,\]
    and
    \item the fourth order Runge-Kutta integrator (RK4) \cite{Lambert:1991}.
\end{itemize}

For the comparison with the damped oscillator with forcing, the symplectic methods are extended in a natural way by additionally adding the forcing term when evaluating the acceleration component in each step.

The error plots in Figures \ref{fig:dsmall}--\ref{fig:flarge} show a regularised relative error computed as follows: if $x_i$ denotes the value of the exact solution at time $t_i$ and $x^*_i$ is the corresponding value of the approximate solution, $\mathrm{err}_i = \frac{10 + x^*_i}{10 + x_i} - 1$.

The simulations have been performed in \texttt{python}, with support from the \texttt{scipy}, \texttt{numpy} and \texttt{matplotlib} libraries.
The plots have been generated using \texttt{matplotlib}, with a style imported from the \texttt{seaborn} library.
All code is released with an MIT license and available from GitHub and Zenodo \cite{zenodoCode}.

In our numerical experiments the fourth order Runge-Kutta method shows an impressive level of accuracy, and at least for this example. The main reason for choosing our method over it is when guarantees of the geometric invariants are more important than the solution accuracy. In all cases under consideration, regardless of the size of the error, our first and second order contact integrators guarantee the conservation of the contact structure, unlike any of the other methods.

As already shown in Example \ref{ex-harmonic-symmetric}, the leapfrog method and our second order contact method are equivalent, except for the initialization of the momentum. 	If $x_j = x(jh)$ for some smooth interpolating curve $x$, then we have
\begin{align*}
    p_0 &= \dot{x}(0) + \frac{h^2}{6} x^{(3)}(0) + \cO(h^4) , \\
    \pi_0 &= \dot{x}(0) + \frac{h^2}{6} x^{(3)}(0) + \frac{h^2}{4} \alpha \ddot{x}(0) + \cO(h^4) ,
\end{align*}
hence if $x^{(3)}(0)$ and $\alpha \ddot{x}(0)$ have the same sign, then it is best to initialize with $p$, i.e.~use the contact method. If they are of opposite sign, which is very likely for overdamped systems, the leapfrog method will be better.

Furthermore, in the limit $\alpha \to 0$, i.e.\@ in the limit of the system becoming symplectic, both the integrators presented in Example \ref{ex-harmonic} and \ref{ex-harmonic-symmetric} converge to the same symplectic leapfrog scheme. The same is true for the time-dependent case of Example \ref{ex-harmonic-time-dep}. Thus for small values of $\alpha$ we a priori expect our integrator to be on par with the leapfrog integrator, and in general perform worse than the third order Ruth3 integrator. As can be seen from Figures \ref{fig:dsmall} and \ref{fig:fsmall}, for $\alpha=0.01$ this is indeed the case: the contact integrators are performing very similarly and Ruth3 performs much better.

One interesting fact, however, is that already for $\alpha=0.1$ our method outperforms the third order Ruth3 method. We believe that the reason for this is that the Ruth3 method is only symplectic for separable Hamiltonians, i.e. if $\dot{p} = f(q)$ and $\dot{q} = g(p)$, whereas with damping, the acceleration depends also on $p$.

Similarly, even though the contact integrator is outperforming the VNC integrator in all the simulations, we see that for $\alpha=0.01$ the VNC integrator and our contact integrators have similar performances, but the contact integrator is performing much better when $\alpha$ increases. In this case, the lack of separability does not explain the difference in performance, and we are left to believe that it is the preservation of the contact structure that makes the difference.

\section{Conclusions}
In this work we have begun a thorough investigation of geometric numerical integrators for contact flows. Contrary to~\cite{Feng2010}, our approach is variational: we discretize Herglotz' variational principle and obtain the discrete version of the generalized Euler-Lagrange equations (see Theorem~\ref{thm-dgEL}).
Furthermore, in Theorems~\ref{thm-contact} and~\ref{thm-inverse} we have proved that the discrete map thus obtained is contact and that any geometric integrator for contact flows must be of this (variational) type.

In Theorem~\ref{thm-modified} we presented a formal backward error analysis for contact variational integrators, showing that the numerical solutions are interpolated by contact flows.

Finally, we have considered the implementation of the first and second order contact integrators for the benchmark example of a damped harmonic oscillator, both with and without external forcing. Our numerical experiments show that contact variational integrators in general are comparable with both symplectic and variational but non-contact (VNC) methods for forced Lagrangian systems of the same order, but that in situations where the contact structure is more relevant (for instance, when the damping increases), they usually outperform them.

Motivated by the results of this work, we expect to extend our analysis in multiple directions.
On the one hand, we plan to derive higher order analogues of the first and second order contact methods presented here.
On the other hand, we would like to
compare our approach with the purely Hamiltonian integrators put forward in~\cite{Feng2010} in a number of systems.
With this future work in mind, we expect that the implementation of contact integrators will be beneficial for the study of a wide range of applications where dissipation plays a central role.

\begin{figure}[p]
    \centering
    \includegraphics[trim=100 0 100 50, clip, width=\linewidth]{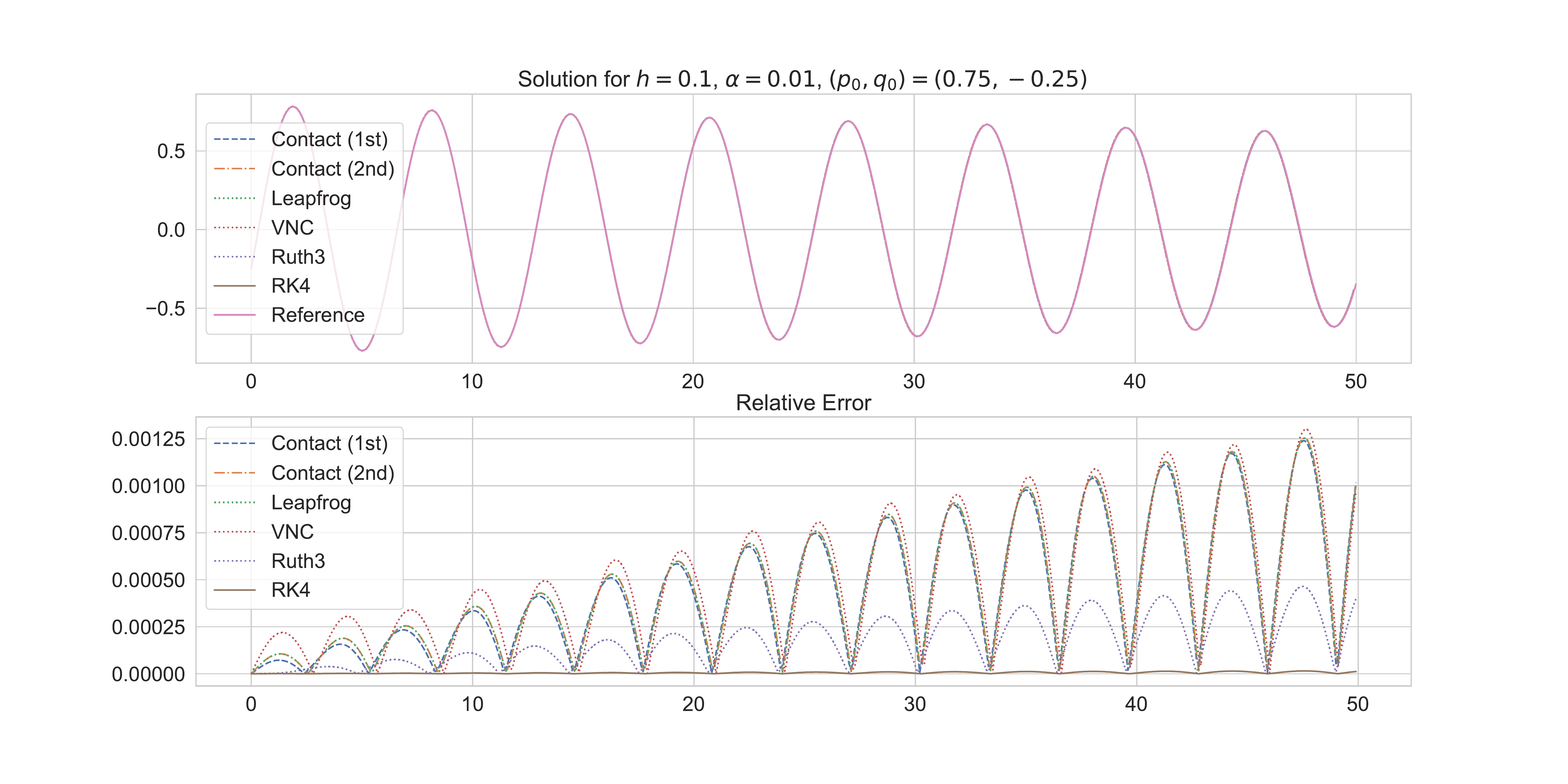}
    \includegraphics[trim=100 50 100 50, clip, width=\linewidth]{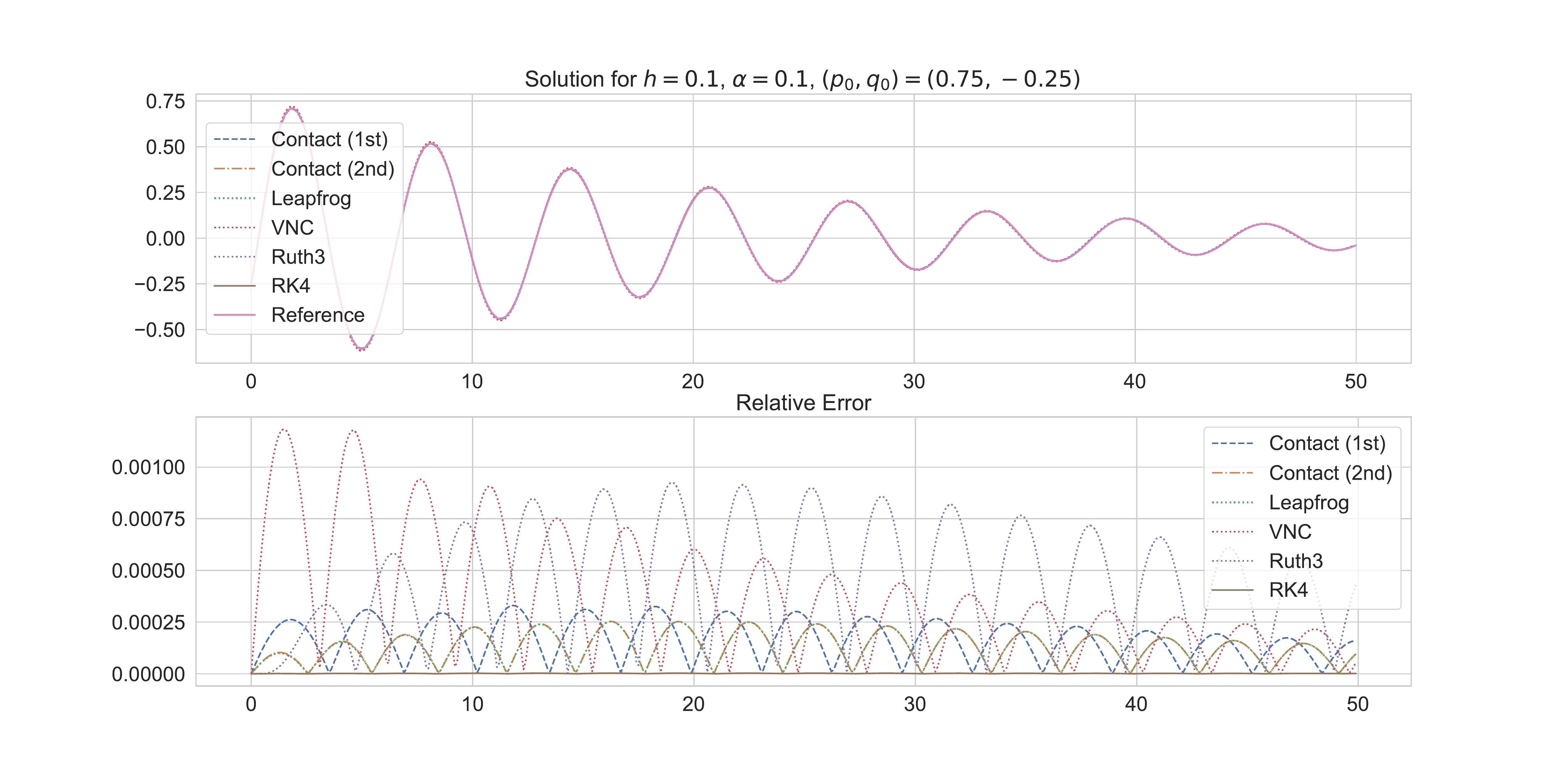}
    \caption{Damped oscillator: comparison to symplectic integrators for small damping parameter $\alpha$}
    \label{fig:dsmall}
\end{figure}

\begin{figure}[p]
    \centering
    \includegraphics[trim=100 50 100 50, clip, width=\linewidth]{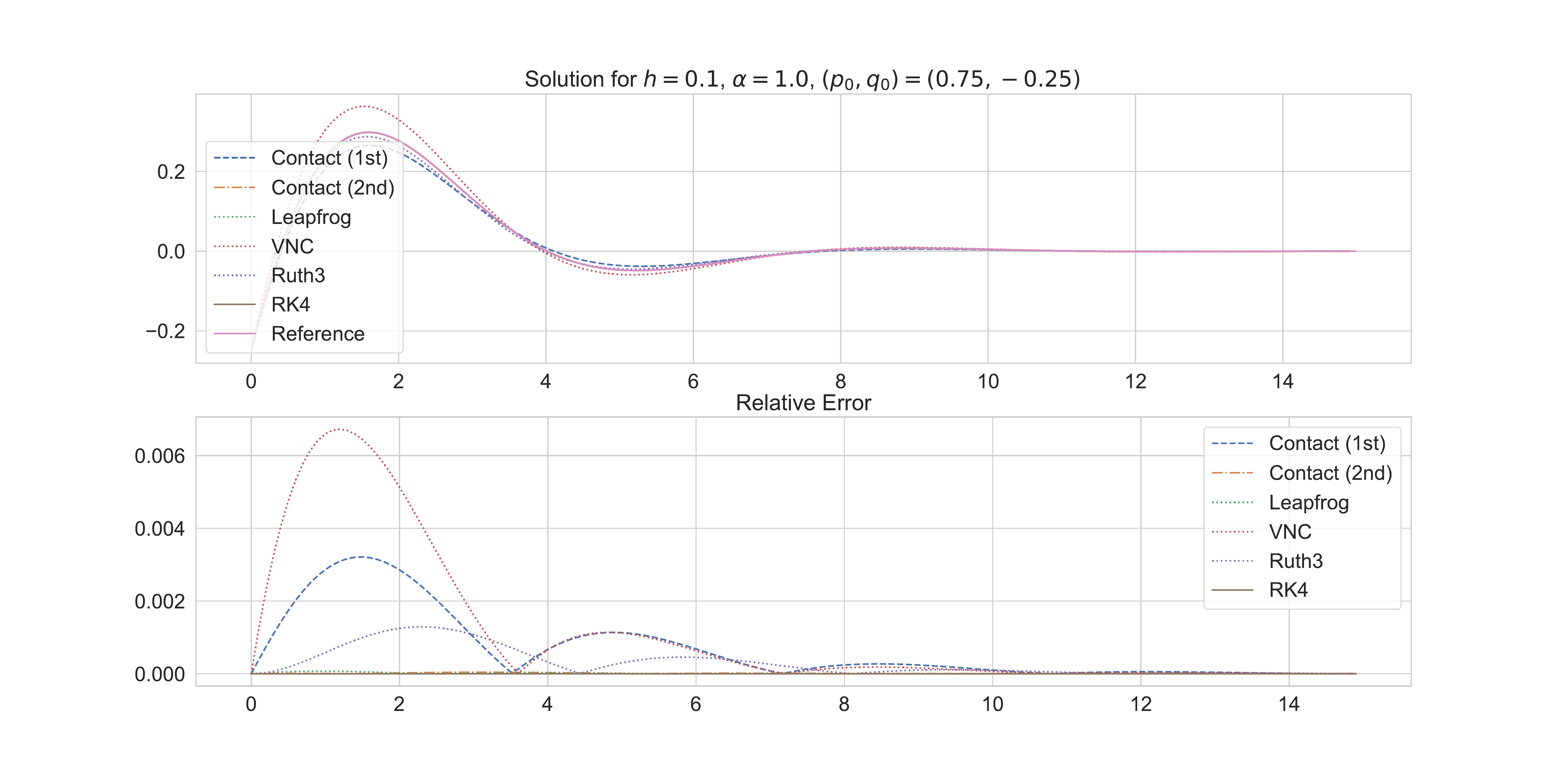}
    \caption{Damped oscillator: comparison to symplectic integrators for critical damping parameter $\alpha$}
    \label{fig:dcrit}
\end{figure}

\begin{figure}[p]
    \centering
    \includegraphics[trim=100 50 100 50, clip, width=\linewidth]{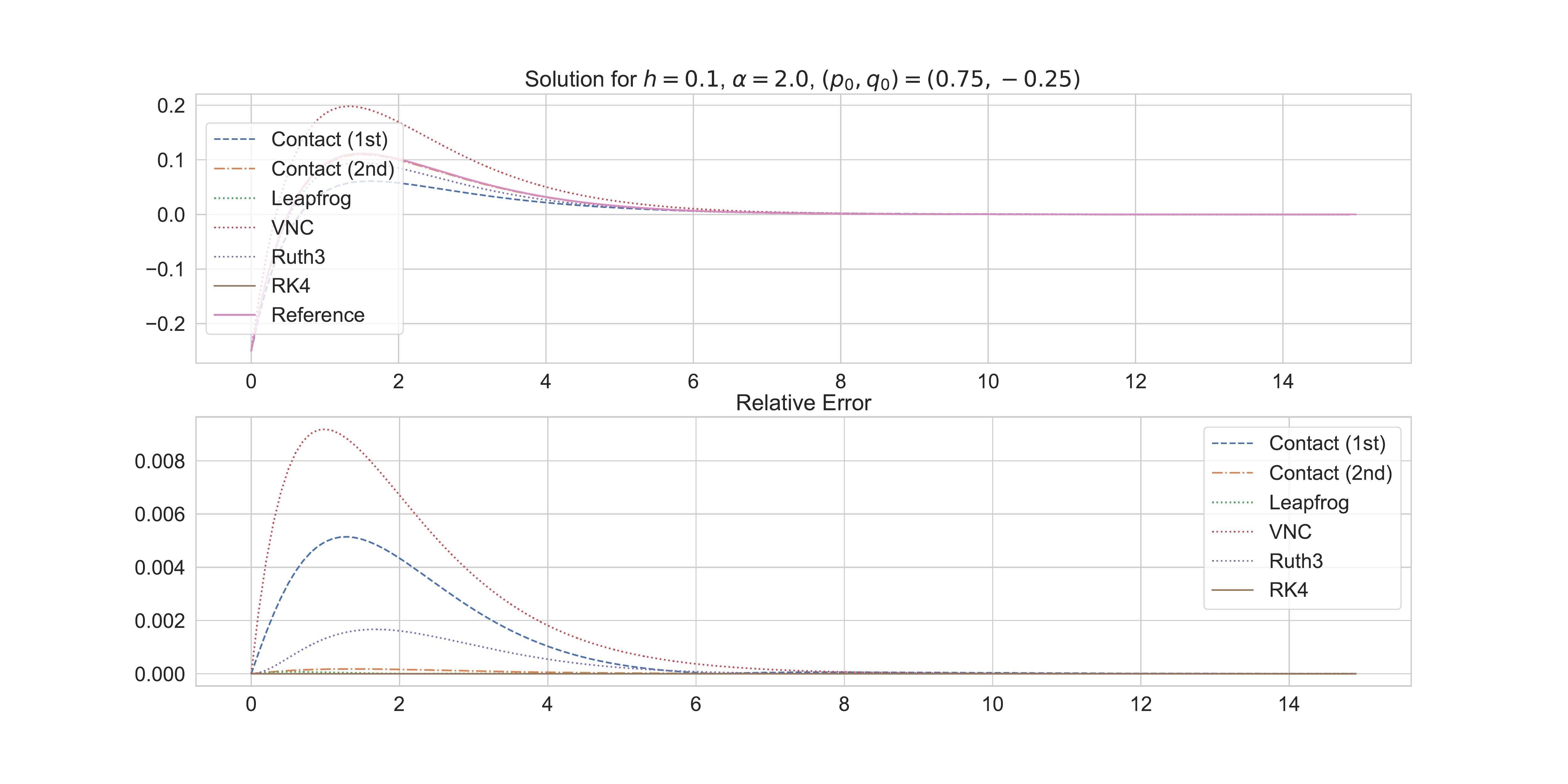}
    \caption{Damped oscillator: comparison to symplectic integrators for larger damping parameter $\alpha$}
    \label{fig:dlarge}
\end{figure}

\begin{figure}[p]
    \centering
    \includegraphics[trim=100 0 100 50, clip, width=\linewidth]{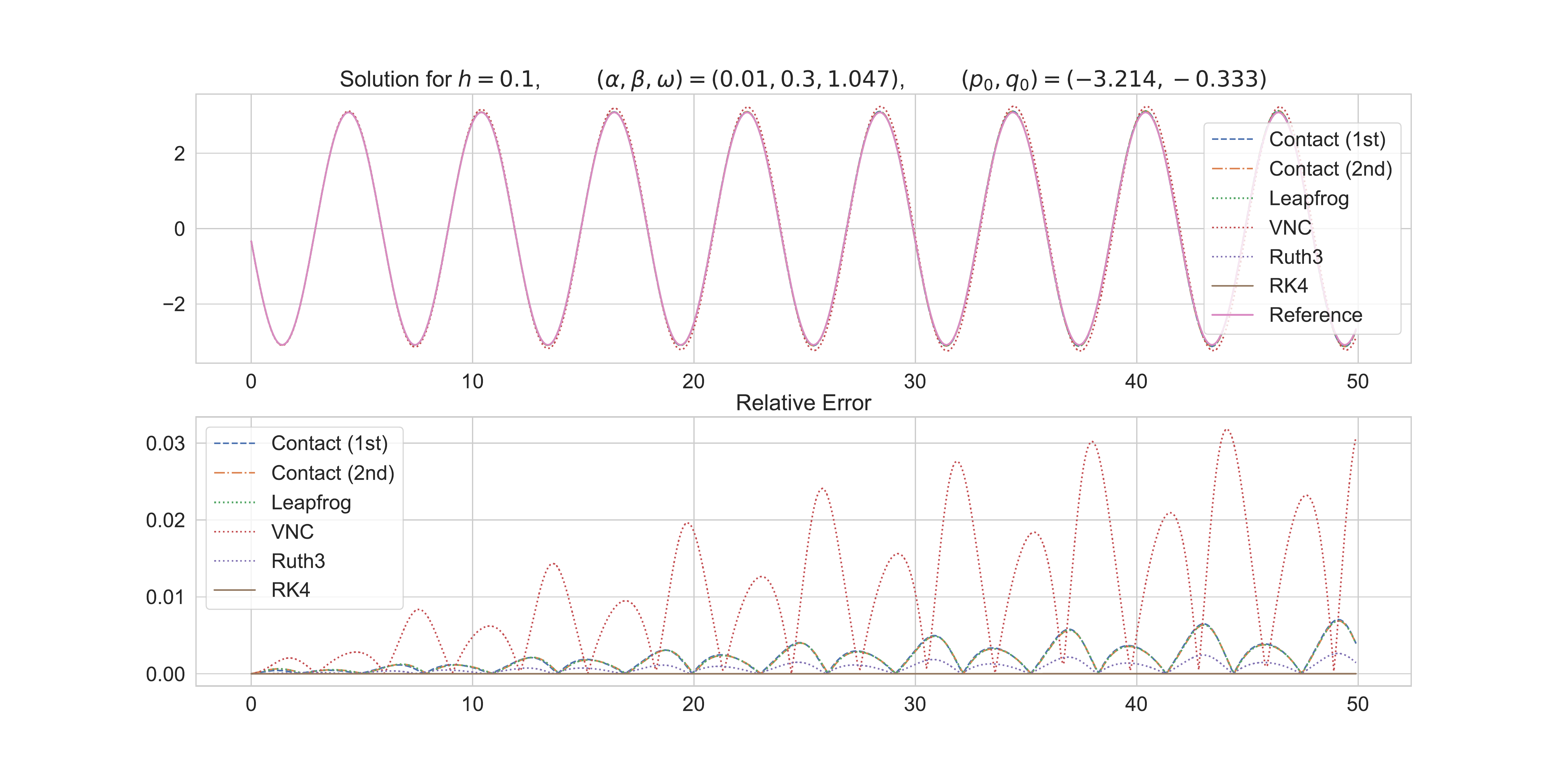}
    \includegraphics[trim=100 50 100 50, clip, width=\linewidth]{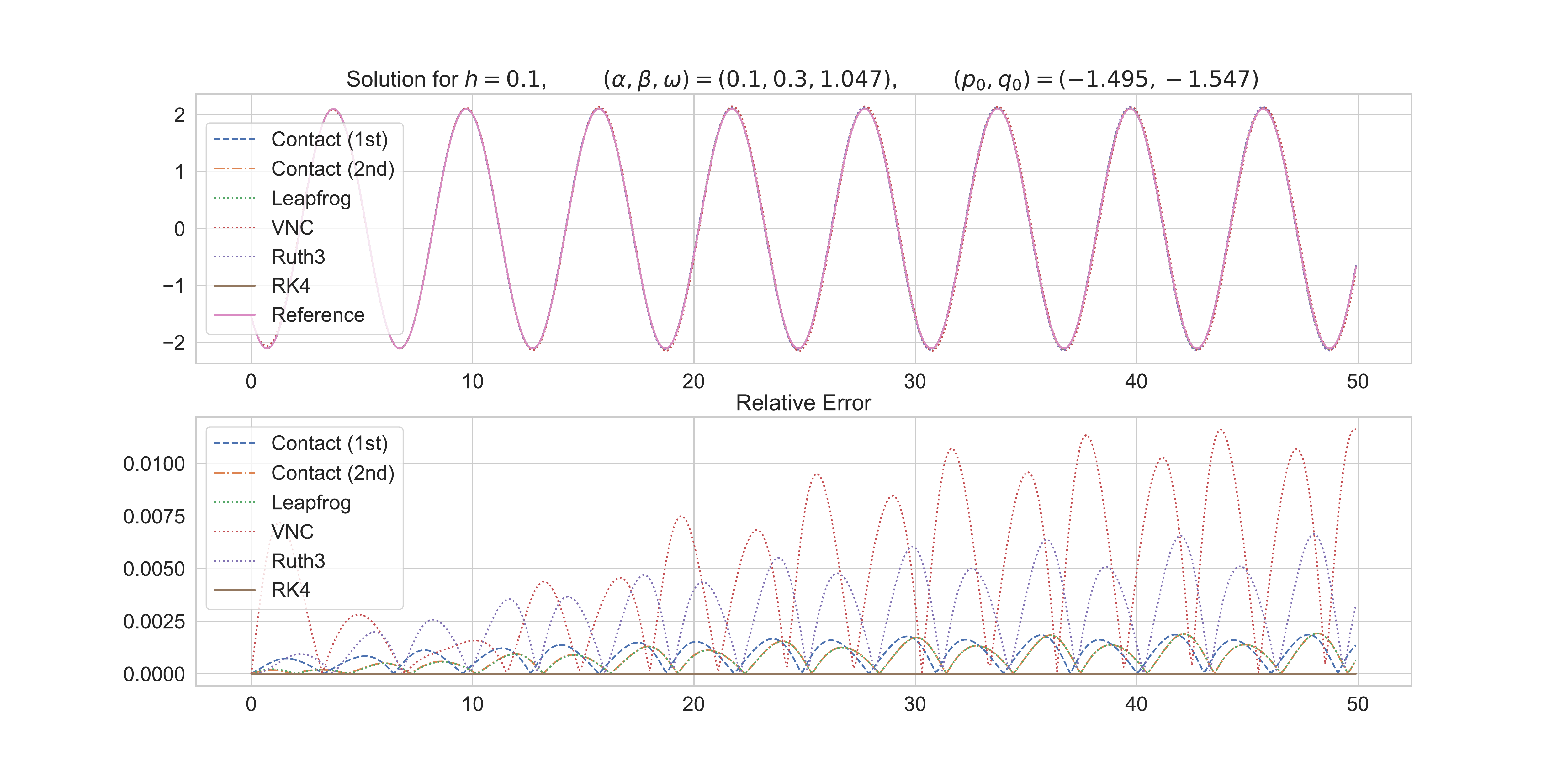}
    \caption{Forced oscillator: comparison to symplectic integrators for small damping parameter $\alpha$ and $f(t) = \beta \sin(\omega t)$}
    \label{fig:fsmall}
\end{figure}

\begin{figure}[p]
    \centering
    \includegraphics[trim=100 50 100 50, clip, width=\linewidth]{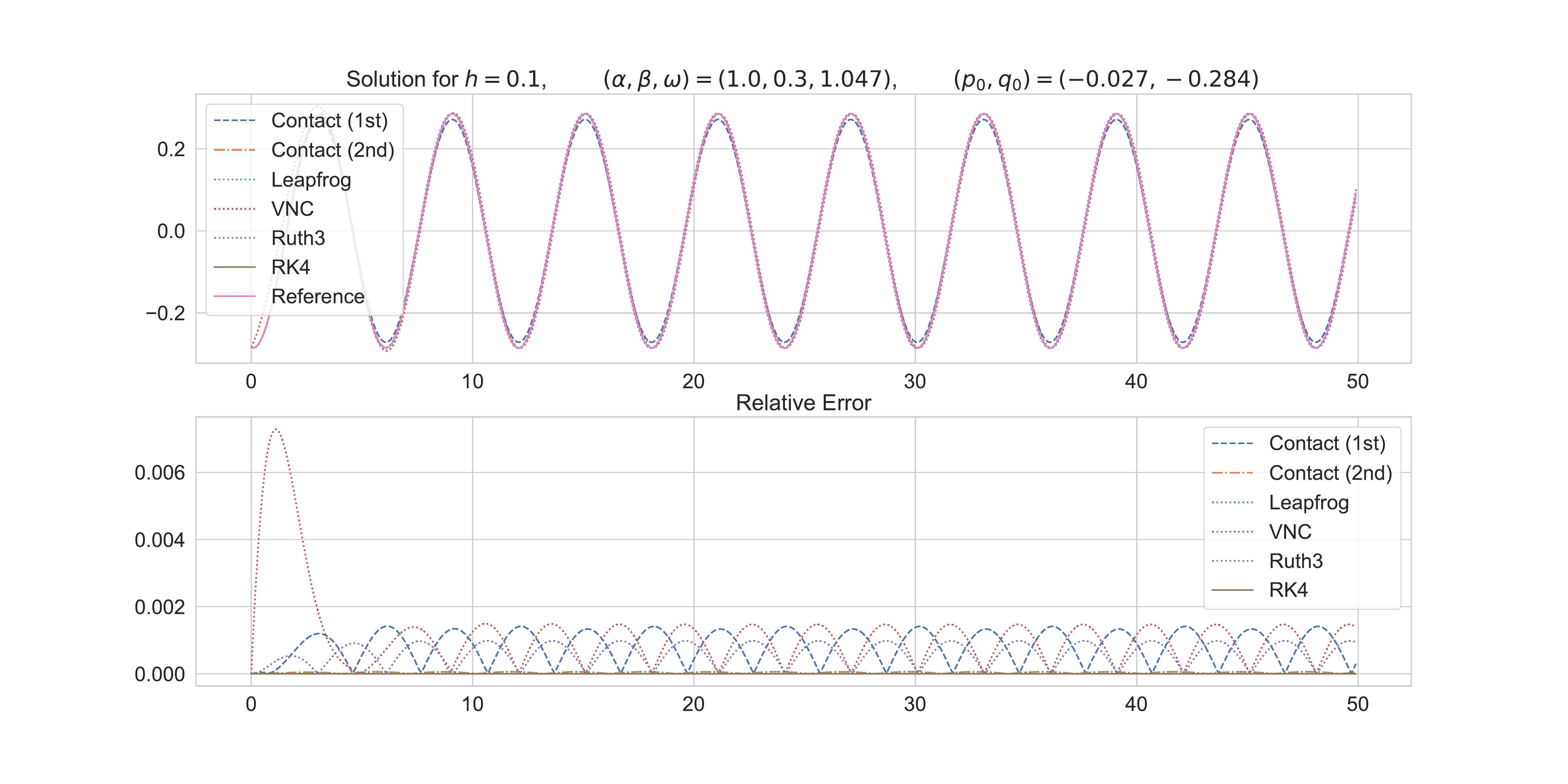}
    \caption{Forced oscillator: comparison to symplectic integrators for critical damping parameter $\alpha$ and $f(t) = \beta \sin(\omega t)$}
    \label{fig:fcrit}
\end{figure}

\begin{figure}[p]
    \centering
    \includegraphics[trim=100 50 100 50, clip, width=\linewidth]{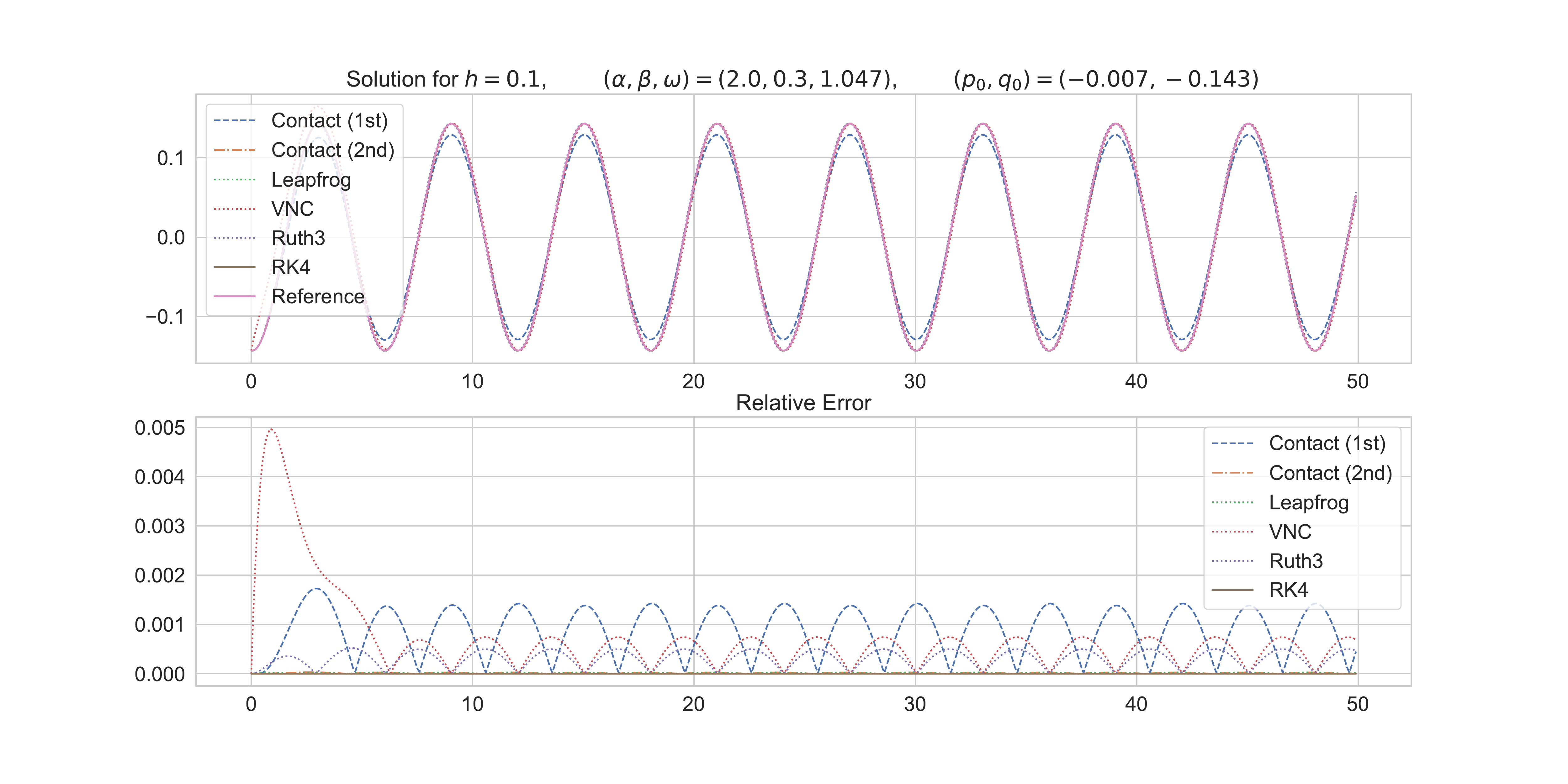}
    \caption{Forced oscillator: comparison to symplectic integrators for larger damping parameter $\alpha$ and $f(t) = \beta \sin(\omega t)$}
    \label{fig:flarge}
\end{figure}

\section*{Acknowledgements}
The authors would like to thank the organizers of the VI Iberoamerican Meeting on Geometry, Mechanics and Control, during which part of this work was initiated, and the NWO Visitor Travel Grant 040.11.698 that sponsored the visit of AB at the Bernoulli Institute. MV is funded by the SFB Transregio 109 ``Discretization in Geometry and Dynamics''. AB acknowledges FORDECYT (project number 265667) for financial support. MS research is supported by the NWO project 613.009.10. The authors would also like to thank the anonymous referees for useful comments that improved the final version of this paper.

\bibliographystyle{abbrvnat_mv}
\bibliography{contact_mv}
\end{document}